\documentclass[11pt]{amsart}
\usepackage{amsmath}
\usepackage{cases}
\usepackage{amssymb, verbatim}
\usepackage{comment}
\usepackage{color}
\usepackage{tikz}

\title[Mean Curvature Flow of Totally Real Submanifolds]{A convergence result for Mean Curvature Flow of totally real submanifolds}

\author[T. C. Collins]{Tristan C. Collins}
  \email{tristanc@math.toronto.edu}
  \address{Department of Mathematics, University of Toronto,  40 St. George St., Toronto, ON, M5S 2E4}
  \email{tristanc@mit.edu}
  \address{Department of Mathematics, Massachusetts Institute of Technology, 77 Massachusetts Avenue, Cambridge, MA 02139}
 \thanks{T.C.C is supported in part by NSF CAREER grant DMS-1944952, and an NSERC Discovery Grant. }

 \author[A. Jacob]{Adam Jacob}
  \email{ajacob@math.ucdavis.edu}
  \address{Department of Mathematics, University of California, Davis, 1 Shields Ave, Davis, CA, 95616}
  \thanks{A.J. is supported in part by a Simons collaboration grant}
  
  \author[Y.-S. Lin]{Yu-Shen Lin}
  \email{yslin@bu.edu}
  \address{Department of Mathematics, Boston University, 665 Commonwealth, Boston, MA 02215}
  \thanks{Y.-S Lin is supported in part by NSF grant DMS-2204109 and a Simons Collaboration grant}

\theoremstyle{plain}
\newtheorem{thm}{Theorem}[section]
\newtheorem{prop}[thm]{Proposition}
\newtheorem{defn}[thm]{Definition}
\newtheorem{lem}[thm]{Lemma}
\newtheorem{cor}[thm]{Corollary}

\theoremstyle{defn}

\newtheorem{rk}[thm]{Remark}

\numberwithin{equation}{section}

\renewcommand{\thanks}[1]{\footnote{#1}} 

\newcommand{\be}{\begin{equation}}
\newcommand{\bea}{\begin{eqnarray}}
\newcommand{\eea}{\end{eqnarray}}
 \newcommand{\ee}{\end{equation}}
 \newcommand{\ba}{\begin{align}}
\newcommand{\ea}{\end{align}}
 
 \def\ba{\begin{eqnarray}}
\def\ea{\end{eqnarray}}

\def\log{\,{\rm log}\,}

\def\[{{\bf [}}
\def\]{{\bf ]}}

\begin{document}

\maketitle

\begin{normalsize}

\begin{abstract}
We establish a convergence result for the mean curvature flow starting from a totally real submanifold which is ``almost minimal" in a precise, quantitative sense.  This extends, and makes effective, a result of H. Li \cite{Li} for the Lagrangian mean curvature flow.
\end{abstract}

\section{Introduction}\label{sec: intro}

The purpose of this paper is to establish a long-time existence and convergence result for the mean curvature flow starting from a totally real, ``almost minimal" submanifold in a K\"ahler-Einstein manifold.  Let  $(X,\bar g,J,\bar\omega)$ be a K\"ahler-Einstein manifold of complex dimension $n$, and $\overline{\nabla}$  the Levi-Civita connection on $X$ with respect to $\bar g$. Let $\lambda$ be the K\"ahler-Einstein constant, i.e.
$$\overline {\rm Ric}=\lambda\overline g.$$
Let $L$ be a real, $n$-dimensional, compact, smoothly immersed submanifold of $X$.
\begin{defn}
An immersion $F:L\rightarrow X$ is totally real if for all $p\in L$
$$F_*(T_p L)\cap JF_*(T_pL)=\{0\}.$$
In other words $F_*(T_pL)$ contains no complex lines.
\end{defn}
A special case of the totally real condition occurs when $L$ is a Lagrangian, defined by $\bar\omega|_L=0$, which implies that the spaces $F_*(T_p L)$ and  $JF_*(T_pL)$ are perpendicular with respect to $\bar g$. Notice that while the Lagrangian condition is closed, the totally real condition is open and, in particular, any small deformation of a Lagrangian submanifold is totally real.  

In the K\"ahler-Einstein case the mean curvature flow  preserves the Lagrangian condition \cite{Sm1}, and if the flow exists for all time and converges, the limit is a minimal Lagrangian.  When the ambient K\"ahler-Einstein manifold is Calabi-Yau, minimal Lagrangians are {\em special Lagrangian} in the sense introduced by Harvey-Lawson \cite{HL}, and their existence is conjectured by Thomas-Yau \cite{TY} to be equivalent to an algebro-geometric notion of stability.  Neves \cite{Neves, Neves1} showed that finite time singularities of the Lagrangian mean curvature flow are unavoidable and Joyce \cite{Joyce} proposed an update of the Thomas-Yau conjecture linking the Langrangian mean curvature flow and Bridgeland stability conditions on the Fukaya category; see \cite{LSS, LSS2} for some recent progress in this direction.

In contrast to the Lagrangian setting,   totally real submanifolds are much less studied, though they have appeared sporadically in the literature for some time; see e.g. \cite{ChOg, Yau}.  Recent works of Pacini \cite{Pacini} and Lotay-Pacini \cite{LP1, LP2, LP3} have initiated a systematic study of totally real submanifolds and their naturally associated notions of volume and minimality, in addition to investigating how these submanifolds behave under the mean curvature flow and related flows such as the Maslov flow.   

Our main result is focused  on the mean curvature flow. Consider the following fixed geometric parameters associated to a  totally real submanifold $L\subset X$:
\begin{itemize}
\item[(i)]{\bf Volume}: ${\rm Vol}(L)= V$.
\item[(ii)]{\bf Pointwise second fundamental form bound}: $\sup_{L} |A|^2 \leq \Lambda$.
\item[(iii)]{\bf Non-collapsing constant}: Fix $\kappa, r_0$ such that, for all $p\in L, r \leq r_0$, we have  $Vol(B_p(r)) \geq \kappa r^n$, where $B_{r}(p)\subset L$ is an intrinsic ball.
\item[(iv)]{\bf Eigenvalue:} Let $\lambda_1^1$ denote the first non-zero eigenvalue of the Hodge Laplacian on $\Omega^1(L)$.
\item[(v)]{\bf Local curvature bounds:} Fix $R>0$ and assume that $\overline{Rm}$, the Riemann tensor of $\bar{g}$, satisfies
$$
\sup_{B_{\Lambda^{-1/2}R}(L)}|\overline{\nabla}^{\ell}\overline{Rm}|^2 \leq\Lambda^{2+\ell} \quad 0 \leq \ell \leq 5.
$$
\end{itemize}

\begin{thm}\label{thm: main}
Let $(X,\bar g,J,\bar\omega)$ be a K\"ahler-Einstein manifold of complex dimension $n$ with K\"ahler-Einstein constant $\lambda$.  Let $L\subset X$ be a smoothly immersed, compact, totally real submanifold, and assume $[\bar\omega]|_{L} = 0 \in H^{2}(L, \mathbb{R})$. Let $H$ denote the mean curvature $1$-form.  With the above notation, there exists a constant $\epsilon$ depending only on $n$  and polynomially on lower bounds for $\frac{\kappa r_0^2}{V}$, $\frac{\lambda^{1}_1-\lambda}{\Lambda}$, $\frac1{ \Lambda^nV^2}$, $ \min\{ 1, \frac{\lambda^{1}_1 -\lambda}{\lambda^1_1}\}$, and $R$, such that if
$$
\sup_L\left( \Lambda^{-1}\big|H\big|^2 + \big|\bar\omega|_{L}\big|^2 \right)\leq \epsilon,
$$
then the mean curvature flow starting from $L$ exists for all time and converges exponentially fast to a minimal Lagrangian submanifold in $B_{\Lambda^{-1/2}R}(L)$.
\end{thm}

\begin{rk}
A few remarks concerning the theorem:
\begin{itemize}
\item In the case that $\lambda \ne 0$ the exactness assumption $[\bar\omega]|_{L} = 0 \in H^{2}(L, \mathbb{R})$ is automatic.  Furthermore, in negative K\"ahler-Einstein case the dependence on $\lambda^1_1$ can be removed.
\item The precise quantitative dependence of $\epsilon$ on $V,\Lambda, \kappa, r_0, \lambda^1_1, \lambda, n$ is given in Remark~\ref{rk: dependence}. The main upshot is that $\epsilon$ is controlled by a polynomial function, of controlled degree, of the scale invariant quantities appearing in Theorem~\ref{thm: main}. Note the quantity $\Lambda^{-1}\big|H\big|^2 + \big|\bar\omega|_{L}\big|^2$ is also scale invariant under dilations of $\overline g$.
\item The initial control of $H, \bar{\omega}_{L}$ can be weakened to $L^2$ control; see Theorem~\ref{thm: L2main}.
\end{itemize}
\end{rk}

When the initial manifold $L$ is Lagrangian, Theorem~\ref{thm: main} was obtained by Li \cite{Li}, who also conjectured the result in the totally real setting.  It is worth remarking that in the Lagrangian setting the dependence is on the first non-zero eigenvalue of the Laplacian on functions, rather than $1$-forms.  This is natural since, according to work of Oh \cite{Oh}, the stability of a minimal Lagrangian under {\em Hamiltonian} perturbations depends on the spectrum of the Laplacian on functions on $L$.  On the other hand, the stability of a minimal Lagrangian under general perturbations depends on the spectrum of the Laplacian on $1$-forms.  We also refer the reader to \cite{KT} for other extensions of Li's result.

Gromov-Lees \cite{Lees} proved an $h$-principle for Lagrangian immerisions.  In the complex dimension two case, this result implies that if $F:L\rightarrow X$ is an immersion into a K\"ahler-Einstein surface such that $F^*[\overline\omega]=0 \in H^2(L,\mathbb{R})$, then $F$ is homotopic to a Lagrangian immersion.  However, such a homotopy may be hard to control, and may not preserve the geometry of $L$.  An example illustrating the possible pathologies of this homotopy is given by the minimal $2$-sphere in the Atiyah-Hitchin manifold.  This $2$-sphere is the globally unique  minimal $2$-sphere \cite{TsaiWangAH}; it is neither Lagrangian, totally real, nor holomorphic.  Interestingly, however, by the result of Gromov-Lees, it is homotopic to a Lagrangian, but this homotopy must be rather wild.  For example, by the remarkable stability theorem of Tsai-Wang \cite{TsaiWangStab, TsaiWangHol}, it cannot be small in $C^1$ topology (or even significantly weaker topologies \cite{LSch}).  Thus, it does not seem straightforward to deduce the existence of a special Lagrangian, as asserted in Theorem~\ref{thm: main}, using  some homotopy or Moser's trick type argument and then applying the Lagrangian mean curvature flow result of Li \cite{Li}. In higher dimensions the Gromov-Lees result is more challenging to apply and, under the assumptions of Theorem~\ref{thm: main}, the authors are not aware of an argument   showing that  an initial totally real $L$ is homotopic to a Lagrangian embedding, aside from the proof presented in this paper.  In particular, utilizing the mean curvature flow in the non Lagrangian case is necessary in certain setups.

In the future we hope to use the convergence results of Theorem~\ref{thm: main} for geometric applications, including the construction of new minimal Lagrangians.  Indeed, in \cite{CJL} we prove a quantitative version of Li's result \cite{Li} allowing for the degeneration of the background geometry.  This enables us to construct new special Lagrangians in certain non-compact Calabi-Yau manifolds. The construction begins by finding Lagrangian submanifolds which are close to minimal, in a quantitative sense, and then running the mean curvature flow.  However, in more general applications the Lagrangian condition  may be difficult to impose, and being able to apply the mean curvature flow in the totally real setting could be an asset. Additionally, these applications  motivate the rather technical looking statement of Theorem~\ref{thm: main}, which is stated precisely to allow the background geometry of $(X, \bar{g})$ to degenerate in some quantitative sense.  

The proof of Theorem~\ref{thm: main} builds on the work of Li \cite{Li} and uses in a crucial way ideas introduced in the setting of the K\"ahler-Ricci flow by Phong-Sturm \cite{PS} and Phong-Song-Sturm-Weinkove \cite{PSSW}.  To highlight this idea, let us recall some basic facts concerning special Lagrangians.  Recall that a fundamental result of Thomas \cite{Thomas} says that special Lagrangians are zeroes of an infinite dimensional moment map; see \cite{LP2} for results in the totally real setting.  Furthermore, by McLean's theorem \cite{McLean}, deformations of special Lagrangians are in one-to-one correspondence with harmonic $1$-forms on $L$.  They key idea of Phong-Sturm is to use lower bounds for eigenvalues of a Laplacian to obtain exponential decay of the log-norm of the moment map along its gradient flow.  In turn, Phong-Sturm explain that the eigenvalue lower bounds are intimately tied to the Hausdorff topology on the moduli space of stable objects for a GIT problem.  In particular, the degeneration of the first non-trivial eigenvalue is closely related to the ``jumping up" of the stabilizer group along the gradient flow, and hence to the failure of GIT stability. In the setting of the Lagrangian mean curvature flow this idea suggests that lower bounds for the first non-trivial eigenvalue of the Laplacian on $1$-forms should be intimately related to the exponential decay of the $L^2$ norm of the mean curvature vector and the convergence of the Lagrangian mean curvature flow.  This is the central idea of the strategy we execute in this paper.  It should be noted that, in comparison with the work of Phong-Stum \cite{PS} and Phong-Song-Sturm-Weinkove \cite{PSSW}, we require much more restrictive assumptions on the initial data.

The primary difficulty in executing this strategy in the totally real setting is the proliferation of error terms arising from the failure of the manifold to be Lagrangian.  The main idea of the current paper can be summarized as saying that, along the mean curvature flow, a totally real submanifold which is initially quantitatively ``approximately Lagrangian" remains quantitatively ``approximately Lagrangian," and in fact decays towards a Lagrangian; see, e.g. Propositions~\ref{L2B} and~\ref{exponential}.  Taken together, these results imply that mean curvature flow inherits many of the properties of the Lagrangian mean curvature flow, allowing us to extend the results of Li \cite{Li}. 

The paper is organized as follows.  In Section~\ref{sec: Background} we collect the necessary background material relevant to the mean curvature flow in the totally real setting.  In Section~\ref{L2section} we obtain $L^2$ estimates for the K\"ahler form restricted to $L$ in terms of the mean curvature, the Einstein constant of $\bar{g}$ and, in Calabi-Yau setting, the spectrum of the Laplacian on $1$-forms on $L$.  As a by-product of these estimates we give a different proof of a result of Lotay-Pacini.  Section~\ref{sec: expDecay} contains the main technical result of the paper concerning the exponential decay of the $L^2$ norm of the mean curvature along the flow.  In Section~\ref{sec: geomQuant} we provide some basic estimates for the evolution of the non-collapsing condition and the first non-zero eigenvalue of the Laplacian on $1$-forms along the flow.  In Section~\ref{sec: convergence} we combine our results to prove Theorem~\ref{thm: main}.  Finally in Section~\ref{sec: L2main} we describe the extension of Theorem~\ref{thm: main} to $L^2$ smallness of the initial data.  
\\
\\
\\
{\bf Acknowledgements}:  We would like to thank  J. Lotay for some helpful discussion and for pointing out the example of the minimal $2$-sphere in the Atiyah-Hitchin manifold. We also thank R. Casals for providing us with background on the result of Gromov-Lees. The authors discussed this problem at the American Institute of Mathematics workshop ``Stability in mirror symmetry,'' and would like to thank the organizers and  the institute for providing a productive research environment.

\section{Background}\label{sec: Background}

In this section we introduce some of the background material needed for our main result. Many of the key identities presented here also play an important role in demonstrating that the Lagrangian condition is preserved along the mean curvature flow, and   are derived in the earlier work of Smoczyk \cite{Sm1}, and later Lotay-Pacini \cite{LP1}. Thus, while we prove some necessary formulas, we direct the reader to these sources for further details.

As above let  $(X,\bar g,J,\bar\omega)$ be a K\"ahler-Einstein manifold of complex dimension $n$ with K\"ahler-Einstein constant $\lambda$. Let $F:L\rightarrow X$ be a compact,  totally real smooth immersion.
Given a point $p\in L$ there exists a splitting of the tangent bundle $T_pX=F_*(T_pL)\oplus JF_*(T_pL)$, with projections $\pi_L:T_p X\rightarrow F_*(T_pL)$ and $\pi_J:T_pX\rightarrow JF_*(T_pL)$ (we abuse notation slightly by not distinguishing $p$ and its image $F(p)$).  Additionally there is  a splitting given by the metric $\bar g$, written $T_pX=F_*(T_pL)\oplus N_pL,$
where $NL$ is the perpendicular space to $F_*(TL)$ determined by $\bar g$.  Let $\pi_T$ and $\pi_\perp$ denote the tangential and normal projections.

Choose local coordinates $(x^1,..., x^n)$ for $L$, and  let
\be
F_i(p):=\frac{\partial F}{\partial x^i}\nonumber
\ee 
denote the pushforward of the tangent vectors to $L$.
 Consider  the   map $N: F_*(T_pL)\longrightarrow N_pL$ defined by $N(v):=\pi_\perp(J(v)). $
Set  $N_i:=N(F_i)$, which gives a basis for the normal bundle. For simplicity let $\langle\cdot,\cdot\rangle$ denote the inner product with respect to the fixed metric $\bar g$ on $X$. We use the notation $g:= F^*\bar g$ for the induced metric on $L$, $\omega:=F^*\bar\omega$ for  the restriction of the K\"ahler form to $L$, and $\nabla$ for the Levi-Civita connection with respect to $g$. In coordinates on $L$ the metric $g$ and two form $\omega$ can be expressed via
\be
g_{ij}=\langle F_i,F_j\rangle\qquad{\rm and}\qquad\omega_{ij}=\langle J(F_i),F_j\rangle.\nonumber
\ee
 Note  $\omega=0$ in the Lagrangian case.

For any vector $v\in T_pL$, we have the orthogonal decomposition 
$$J(v)=\pi_\perp(J(v))+\pi_T(J(v)).$$
Since $F_\ell$ form a basis for $T_pL$, in coordinates we write $\pi_T(J(F_i))=b_i{}^\ell F_\ell$  for some coefficients $b_i{}^\ell$. In fact, $b_i{}^\ell=\omega_i{}^\ell$, since
$$b_i{}^m g_{mk}=\langle b_i{}^m F_m, F_k\rangle=\langle\pi_T(J(F_i)), F_k\rangle=\langle J(F_i), F_k\rangle=\omega_{ik}.$$
Here we are using the convention that when working on $L$ we can raise or lower indices with respect the metric $g$. Now, the above shows  $\pi_T(J(v))$ vanishes if $L$ is a Lagrangian, in addition to the useful decomposition  formula
\be
\label{Jtangent}
J(F_j)=N_j+\omega_j{}^pF_p.
\ee

Next  we turn to the restriction of the metric $\overline g$ to the normal bundle, which we denote by:
\begin{align}
\eta_{ij}:=\langle N_i,N_j\rangle&=\langle J(F_i)-\omega_i{}^\ell F_\ell,J(F_j)-\omega_j{}^pF_p\rangle\nonumber\\
&=g_{ij}-\omega_i{}^\ell\omega_{j\ell}-\omega_j{}^p\omega_{ip}+\omega_i{}^\ell\omega_j{}^pg_{\ell p}\nonumber\\
&=g_{ij}-\omega_j{}^p\omega_{ip}=g_{ij}+\omega_j{}^p\omega_{pi}.\nonumber
\end{align}
We remark that in this case, and this case only, raised indices $\eta^{ij}$ will denote the inverse of the metric $\eta_{ij}$, and not a contraction with the metric $g$. To see how $J$ acts on the normal bundle, take $J$ of both sides of \eqref{Jtangent} which gives $J(N_j)=-F_j-\omega_j{}^p N_p-\omega_j{}^p\omega_p{}^q F_q.$ In addition, the formula for $\eta_{jk}$ implies $\eta_{jr}g^{rq}=\delta_j{}^q+\omega_j{}^p\omega_p{}^q$. Thus
\be
\label{Jnormal}
J(N_j)=-\eta_{jr} g^{rp}F_p-\omega_j{}^p N_p.
\ee

\subsection{Extrinsic curvature terms}
Here we provide some useful formulas involving the mean curvature of $L$.
\begin{defn}
The second fundamental form on $L$ is defined by 
$$A_{jk}=\pi_{\perp}(\overline{ \nabla}_{F_j} F_k).$$
We denote the individual  components by
\be
h_{ijk}:=-\langle N_i,\overline{ \nabla}_{F_j} F_k\rangle,\nonumber
\ee
which allows us to write $A_{jk}=-\eta^{ip}h_{ijk}N_p$. 
\end{defn}
Because $\overline\nabla$ is torsion free we see right away that $h_{ijk}=h_{ikj}$. However, the first indice does not commute with the other two if $L$ is not a Lagrangian. This can be seen explicitly by Proposition 1.1 from \cite{Sm1}, stated here:
\begin{prop}
\label{derivativeo}
Commuting the first two indices of $h_{ijk}$ gives
\be
h_{ijk}-h_{jik}=\nabla_k\omega_{ji}.\nonumber
\ee
\end{prop}

Next we turn to the mean curvature. For a vector $v\in T_pL$, recall the shape operator is defined by $S_{N_i}(v)=-\pi_T(\overline{\nabla}_{v}N_i)$, which we compute  as follows. For any basis $\{e_\alpha\}$ of $T_pX$, we can write the vector $\overline{\nabla}_{F_j}N_p$ as
\be
\overline{\nabla}_{F_j}N_p=\langle \overline{\nabla}_{F_j} N_p, e_\alpha\rangle e_\beta \overline{g}^{\alpha\beta}.\nonumber
\ee
With the  basis $F_1,...,F_n, N_1,...,N_n$,  note that $ \overline{g}$ is block diagonal in the tangental and normal directions to $L$. This gives
\be
\label{thing5}
\overline{\nabla}_{F_j}N_p=\langle \overline{\nabla}_{ F_j} N_p, F_q\rangle F_s g^{qs}+\langle \overline{\nabla}_{F_j} N_p, N_q\rangle N_s \eta^{qs}.
\ee
In particular this implies
$$\pi_T(\overline{\nabla}_{F_j}N_p)=\langle \overline{\nabla}_{ F_j} N_p, F_q\rangle F_s g^{qs}=-\langle N_p,\overline{\nabla}_{ F_j}  F_q\rangle F_s g^{qs}=h_{pjq} F_s g^{qs}.$$
So $S_{N_p}(F_j)=-h_{pjq}g^{qs} F_s$, or $(S_{N_p})^s{}_j=-h_{pjq}g^{qs}.$  We then see the trace of the shape operator is  ${\rm Tr}(S_{N_p})=-h_{pjq}g^{qj}$. For notational simplicity, we let $H_p:=g^{jq}h_{pjq}$, so  ${\rm Tr}(S_{N_p})=-H_p$.

Now, if $\{E_i\}$ is an orthonormal basis   for $N_pL$, the mean curvature vector $\vec H$ is defined by $\vec H=\sum_i {\rm Tr}(S_{E_i})E_i.$
If we instead use the basis $\{N_1,..., N_n\}$, because it is not orthonormal one has
$$\vec H=\eta^{mp} {\rm Tr}(S_{N_m})N_p=-\eta^{mp}H_mN_p.$$
We sometimes denote the components  by $\hat H^p:=-\eta^{mp}H_m$.

The mean curvature vector $\vec H$ can be identified with a $1$-form on $L$ via
$$H:=F^*(i_{\vec H}\overline{\omega}),$$
which we call the mean curvature $1$-form. To write down $H$ in indices, note
$$ H(\frac{\partial}{\partial x^i})=\overline{\omega}(\vec H, F_i)=-\langle \vec H, JF_i\rangle=\langle \eta^{mp}H_mN_p, N_i\rangle=\eta^{mp}H_m\eta_{pi}=H_i,$$
and so $H=H_i dx^i=g^{jk}h_{ijk}dx^i$.

Following the notation of \cite{LP1}, we also define a $1$-form  $\xi:=g^{kj}h_{jik}dx^i$, which is given by a different trace of the second fundamental form than $H$. In the Lagrangian case $\xi=H$. In the general totally real case, by Proposition \ref{derivativeo} we have  $-\nabla_k\omega_{ji}=h_{jik}-h_{ijk},$ which implies
\be
\label{xiH}
d^*\omega =-g^{kj}\nabla_k\omega_{ji}dx^i=g^{kj}(h_{jik}-h_{ijk})dx^i=\xi-H.
\ee
This equality, which is Proposition 2.5 in \cite{LP1}, will play a significant role throughout the paper.

Taking the the exterior derivative of \eqref{xiH} yields
\be
\label{dH}
dH=d\xi-dd^*\omega.
\ee
Let $\overline\rho$ be the Ricci form on $X$, and $\rho$ its pullback onto $L$. We now restate Proposition 2.8 from \cite{LP1}, which uses \eqref{dH} to achieve a useful formula for $dH$ (we direct the reader to \cite{LP1} for details).
\begin{prop}
\label{formuladH}
The exterior derivative of $H$ can be expressed as $$dH=\Delta\omega+ \rho+\omega*Rm+ \eta^{-1}*\omega*{\overline {Rm}}+\eta^{-1}*\omega*h*h .$$ \end{prop}
Here we use the standard notation $``*"$ for an expression involving  tensors where the exact contractions and constants do not affect the argument. We also use the ``analyst'' convention to define the Laplacian $\Delta=g^{j  k}\nabla_j\nabla_k$. Note that in our case, because $(X,\omega)$ is  K\"ahler-Einstein, in the above formula we can replace $\rho$ by $\lambda\omega$.

\subsection{The induced connection on the normal bundle}
\label{normalsection}

In addition to computing the tangential components of $\overline{\nabla}_{F_j}N_i$, which define the shape operator $S_{N_i}(F_j)=-\pi_T(\overline{\nabla}_{F_j}N_i)$, it will be useful for us to understand   the normal components  as well.

\begin{prop} 
\label{normal} The normal components of $\overline{\nabla}_{F_j} N_p$ are given by
\be
\pi_\perp(\overline{\nabla}_{F_j} N_p)= \left(\omega_p{}^rh_{qjr}\eta^{qs} -\omega_q{}^rh_{rjp}\eta^{qs}+\Gamma_{jp}^s\right)N_s.\nonumber
\ee
\end{prop}
\begin{proof}
Returning to \eqref{thing5}, we see right away that
\bea
\pi_\perp(\overline{\nabla}_{F_j} N_p)&=& \langle \overline{\nabla}_{F_j} N_p, N_q\rangle N_s \eta^{qs}\nonumber\\
&=& \langle \overline{\nabla}_{F_j}(J(F_p)-\omega_p{}^rF_r), N_q\rangle N_s \eta^{qs}\nonumber\\
&=& \langle  J (\overline{\nabla}_{F_j}F_p), N_q\rangle N_s \eta^{qs}-\omega_p{}^r\langle \overline{\nabla}_{F_j}F_r, N_q\rangle N_s \eta^{qs}\nonumber
\eea
since $J$ is covariant constant and $F_n, N_q$ are orthogonal.  The last term on the right above is a second fundamental form term. Applying \eqref{Jtangent} to  $N_q$ in the other term we see
\bea
\pi_\perp(\overline{\nabla}_{F_j} N_p)&=& \omega_p{}^rh_{qjr} N_s \eta^{qs}+\langle  J (\overline{\nabla}_{F_j}F_p), J(F_q)-\omega_q{}^rF_r\rangle N_s \eta^{qs}\nonumber\\
&=& \omega_p{}^rh_{qjr} N_s \eta^{qs}+\Gamma_{jp}^ug_{uq}N_s \eta^{qs}-\langle  J(\overline{\nabla}_{F_j}F_p), \omega_q{}^rF_r\rangle N_s \eta^{qs}\nonumber\\
&=& \omega_p{}^rh_{qjr} N_s \eta^{qs}+\Gamma_{jp}^ug_{uq}N_s \eta^{qs}+\omega_q{}^r\langle \overline{\nabla}_{F_j}F_p, J(F_r)\rangle N_s \eta^{qs}.\nonumber
\eea
One more application of \eqref{Jtangent} gives
\begin{align}
\pi_\perp(\overline{\nabla}_{F_j} N_p)&= \omega_p{}^rh_{qjr} N_s \eta^{qs}+\Gamma_{jp}^ug_{uq}N_s \eta^{qs}+\omega_q{}^r\langle  \overline{\nabla}_{F_j}F_p, N_r+\omega_r{}^\ell F_\ell\rangle N_s \eta^{qs}\nonumber\\
&= \omega_p{}^rh_{qjr} N_s \eta^{qs}-\omega_q{}^rh_{rjp}N_s \eta^{qs}+\Gamma_{jp}^k(g_{kq}+\omega_q{}^r\omega_{rk})N_s \eta^{qs}.\nonumber
\end{align}
The term in parenthesis on the right is  $\eta_{qk}$, which completes the proof.

\end{proof}

Let $\vec V=V^p N_p$ be a section of the normal bundle $NL$. Pulling this bundle back  as $F^*(NL)$, we let $\hat\nabla$ denote the pullback of  $\pi_\perp\circ\overline{\nabla}$. Specifically
\bea
(\hat\nabla_j V^s)N_s&:=& \pi_\perp\circ\overline{\nabla}_{F_j}(V^sN_s) =F_j(V^s)N_s+V^p \pi_\perp(\overline{\nabla}_{F_j} N_p)\nonumber\\
&=&\left(F_j(V^s)+V^p(\omega_p{}^rh_{qjr}\eta^{qs} -\omega_q{}^rh_{rjp}\eta^{qs}+\Gamma_{jp}^s)\right)N_s.\nonumber
\eea
For notational simplicity we write 
\be
\label{Gammadef}
\hat\Gamma_{jp}^s:=\eta^{sq}\langle\overline{\nabla}_{F_j} N_p, N_q\rangle=\omega_p{}^rh_{qjr}\eta^{qs} -\omega_q{}^rh_{rjp}\eta^{qs}+\Gamma_{jp}^s.
\ee
Working on $L$, we then have
\be
\label{hatconnection}
\hat\nabla_j V^s=\frac{\partial}{\partial x^j} V^s+\hat\Gamma_{jp}^s V^p.
\ee
Note that $\hat\Gamma_{jp}^s$ is not symmetric in $j$ and $p$, which we would expect, as $j$ corresponds to a derivative in the tangent direction, and $p$ the normal vector, which do not interchange.

Notice that every section $\alpha\in T^*L$  can be identified with a section of the normal bundle $F^*(NL)$ by contracting with the metric $\eta_{ij}$. Specifically, for $\vec V\in F^*(NL)$, we can define a 1-form   $F^*(i_{\vec V}\overline \omega)$  exactly as we did for the mean curvature one form $H$:
$$F^*(i_{\vec V}\overline \omega)=\overline{\omega}(\vec V, F_i)dx^i =-\langle \vec V, JF_i\rangle dx^i=-V^p\eta_{pi} dx^i.$$
From this it is clear that given $\alpha=\alpha_i dx^i$, the contraction $-\eta^{ki}\alpha_i N_k$ defines a section of the normal bundle. Thus it is reasonable to extend the connection $\hat\nabla$ to $T^*L$ by the formula
\be
\label{hatconnectionform}
\hat\nabla_j\alpha_i=\frac{\partial}{\partial x^j} \alpha_i-\hat\Gamma^p_{ji}\alpha_p.
\ee
Note that 
\bea
\partial_k\eta_{ij}&=&\langle \overline\nabla_{F_k}N_i,N_j\rangle+\langle N_i,\overline\nabla_{F_k}N_j\rangle\nonumber\\
&=&\langle \pi_\perp(\overline\nabla_{F_k}N_i),N_j\rangle+\langle N_i,\pi_\perp(\overline\nabla_{F_k}N_j)\rangle\nonumber\\
&=&\hat\Gamma^{s}_{ki}\eta_{sj}+\hat\Gamma^s_{kj}\eta_{is},\nonumber
\eea
from which it follows that $\hat\nabla_k\eta_{ij}=0.$ One could derive \eqref{hatconnectionform} from \eqref{hatconnection} just using contraction with the metric and this fact.

Finally, we see
\be
\label{derivchange}
\hat\nabla_j\alpha_i- \nabla_j\alpha_i=(-\omega_i{}^rh_{qjr}\eta^{qs}  +\omega_q{}^rh_{rji}\eta^{qs})\alpha_s.
\ee
Thus, later on we will be able to change between $\hat\nabla$ and $\nabla$ at the cost of including a term of the form $\omega*h*\eta^{-1}$.

  \subsection{The Maslov 1-form }
  
Let $\{F_1,...,F_n\}$ be a positively oriented basis of $F_*(T_pL)$, which can be extended via $\{F_1,...,F_n, JF_1,..., JF_n\}$ to a basis for $T_pX$. Let $\{F_1^*,...,F^*_n, (JF_1)^*,..., (JF_n)^*\}$ denote the corresponding dual basis. At $p$ define
$$\Omega_J(p):=\frac{(F_1^*+i(JF_1)^*)\wedge\cdots\wedge(F_n^*+i(JF_n)^*)}{|(F_1^*+i(JF_1)^*)\wedge\cdots\wedge(F_n^*+i(JF_n)^*)|_h},$$ 
  where $h$ is the induced metric on $F^*K_X$. This formula defines a section of $F^*K_X$, which is independent of the choice of basis (see Section 3.1 in \cite{LP1}).

A short calculation (Proposition 4.3 in \cite{LP1}) shows
$$\overline\nabla_{F_i}\Omega_J=i\langle J\circ \pi_J(\overline\nabla_{F_i}F_j),F_k\rangle g^{jk}\Omega_J.$$
\begin{defn}
The Maslov 1-form on $L$ is defined by
$$\xi_J=\langle J\circ \pi_J(\overline\nabla_{F_i}F_j),F_k\rangle g^{jk} dx^i.$$ 
\end{defn}
In particular  $\overline\nabla \Omega_J=i\,\xi_J\otimes\Omega_J$. It will be useful to have an expression for $\xi_J$ in local coordinates. Recall that 
   \begin{align} 
\overline\nabla_{F_i}F_j=&\langle  \overline\nabla_{F_i}F_j,F_k\rangle g^{k\ell}F_\ell+\langle \overline\nabla_{F_i}F_j, N_k \rangle \eta^{k\ell}N_\ell \nonumber\\
=&\Gamma^\ell_{ij}F_\ell-\eta^{k\ell}h_{kij}N_\ell\nonumber\\
 =&(\Gamma^s_{ij} +\eta^{k\ell}h_{kij} \omega_{\ell}{}^s)F_s-\eta^{ks}h_{kij}JF_s.\nonumber
\end{align} 
We then see that $J\circ \pi_J(\overline\nabla_{F_i}F_j)=\eta^{ks}h_{kij}F_s,$ which gives
$$\xi_J=\eta^{jk}h_{kij}dx^i.$$

We now restate Proposition  4.4  from \cite{LP1}. Because we are  in the K\"ahler setting, we also use  the first equation from Section 4.3  from that reference.
\begin{prop}
\label{dxiJ}
The exterior derivative of the Maslov 1-form on $L$ is equal to the Ricci form restricted to $L$, i.e.
$$d\xi_J=\rho.$$
\end{prop}

Now, suppose $(X,\overline g, \overline \omega, \Omega)$ is a Calabi-Yau manifold, with $\Omega$ the holomorphic $(n,0)-$form. Restricting to $L$, $\Omega$ is a constant multiple (of modulus 1) of the canonical section $\Omega_J$ of $F^*K_X$. In particular, there exists a well defined angle function $e^{i\theta_L}:L\rightarrow S^1$ given by
$$\Omega|_{L}=e^{i\theta_L}\Omega_J|_{L}.$$
\begin{lem}[Proposition 8.4 in \cite{LP1}]
If $X$ is Calabi-Yau, then the Maslov 1-form can be expressed via
$$\xi_J=-d\theta_L.$$
\end{lem}
\begin{proof}
Restricting to $L$, and using that the holomorphic volume form is covariant constant, we compute
   \begin{align}
   0=\overline\nabla \Omega=& d(e^{i\theta_L})\otimes \Omega_J+e^{i\theta_L}\overline\nabla \Omega_J\nonumber\\
   =& i d\theta_L\otimes e^{i\theta_L} \Omega_J+i\xi_J\otimes e^{i\theta_L} \Omega_J.\nonumber
   \end{align}
   The lemma  follows.
\end{proof}
We remark that the Maslov 1-form of a Lagrangian is sometimes defined by $d \theta_L$, which has the oppose sign as $\xi_J$.

Next we turn to the difference between $\xi$ and $\xi_J$. Recall that the induced metric on the normal bundle is $\eta_{ij}=g_{ij}+\omega_{ip}g^{pq}\omega_{qj}$. Contracting with $g^{-1}$ gives  $\eta_{ij}g^{jm}=\delta_{i}{}^m+\omega_{ip}g^{pq}\omega_{qj}g^{jm},$ and contracting again with $\eta^{-1}$ gives $g^{\ell m}-\eta^{\ell m}=\eta^{\ell i}\omega_{ip}g^{pq}\omega_{qj}g^{jm}$. This allows us to conclude
\be
\label{error}
\xi-\xi_J=(g^{jk}-\eta^{jk})h_{jik}dx^i=\eta^{j s}\omega_{sp}g^{pq}\omega_{qr}g^{rk}h_{jik}dx^i,
\ee
giving a precise expression for the error term.

\section{$L^2$ Control of $\omega$}

\label{L2section}

The main goal of this section is to prove $L^2$ control of $\omega:=\overline\omega|_L$ by the $L^2$ norm of the mean curvature $1$-form $H$. All norms and inner products in this section will be taken with respect to the induced metric $g$ on $L$.

\begin{prop}
\label{L2B}
Let $(X,\overline g, \overline \omega)$ be a K\"ahler-Einstein manifold with Einstein constant $\lambda$. Let $L$ be a totally real submanifold of $X$.  If $\lambda=0$ assume in addition that $[\omega] =0 \in H^{2}(L,\mathbb{R})$. Let  $\sup_L|A|^2=\Lambda$, denote by $\lambda_1^1$ the first non-zero eigenvalue of the Hodge Laplacian on $\Omega^1(L)$, and let
\begin{align}
a&= -\lambda \quad \text{ if } \lambda <0, \nonumber\\
a&=\lambda_1^1-\lambda \text{ if } \lambda\geq 0,\nonumber
\end{align}
Then, if
$$
\sup_L|\omega|^2 \leq \min\left\{\frac{\lambda^1_1-\lambda}{\lambda^1_1}, 1\right\}\frac{a}{4\Lambda},
$$
we have
$$a\int_L|\omega|^2 \leq 4\max\left\{1,\frac{\lambda_1^1}{\lambda^1_1-\lambda}\right\}\int_L |H|^2.$$

\end{prop}
\begin{proof}
First, recall that $\xi-H=d^*\omega$. Also, note that Proposition \ref{dxiJ} gives  $d\xi_J=\rho=\lambda\omega$. In particular, if $\lambda \ne 0$, then $\omega$ is exact as well. We now have 
$$\int_L\langle \xi_J,d^*\omega\rangle_g=\int_L\lambda|\omega|^2.$$
Using that $\xi_J=\xi-(\xi-\xi_J)=d^*\omega+H-\xi+\xi_J$ we see
\begin{align}
-\lambda \int_L|\omega|^2+\int_L|d^*\omega|^2&=\int_L\langle\xi-\xi_J-H, d^*\omega\rangle_g\nonumber\\
&\leq \int_L|\xi-\xi_J-H||d^*\omega|\nonumber\\
&\leq \frac{1}{2\epsilon^2}\int_L|\xi-\xi_J-H|^2+\frac{\epsilon^2}{2}\int_L|d^*\omega|^2,\nonumber
\end{align}
where $\epsilon>0$ is to be determined. Rearranging terms gives
\begin{align}
-\lambda \int_L|\omega|^2+\frac{2-\epsilon^2}{2}\int_L|d^*\omega|^2&\leq  \epsilon^{-2}\int_L|H|^2+ \epsilon^{-2}\int_L|\xi-\xi_J|^2. \nonumber
\end{align}
Since $\omega$ is exact in all cases we may write $\omega=d\beta$ where $\beta$ is a co-exact $1$-form, by the Hodge decomposition. We then see $|d^*\omega|^2=|d^*d\beta|^2=|\Box\beta|^2$, and hence, by a standard spectral argument,
$$
\int_{L}|d^*\omega|^2 =\int_{L}|\Box\beta|^2 \geq  \lambda^1_1\int_{L}|d\beta|^2=\lambda^1_1\int_{L}|\omega|^2.
$$
We arrive at
$$
\left(\frac{2-\epsilon^2}{2}\lambda^1_1 - \lambda\right) \int|\omega|^2 \leq \epsilon^{-2}\int_L|H|^2+ \epsilon^{-2}\int_L|\xi-\xi_J|^2.
$$
Choose $\epsilon^2 =\min\{\frac{\lambda^1_1-\lambda}{\lambda^1_1}, 1\}$, which gives
$$
\frac{a}{2} \int|\omega|^2 \leq \epsilon^{-2}\int_L|H|^2+ \epsilon^{-2}\int_L|\xi-\xi_J|^2.
$$
Given our explicit formula for the error term \eqref{error} we see it is quadratic in $\omega$ with one second fundamental form term, and so
\begin{align}
\label{thing200}
\frac{a}{2} \int|\omega|^2 \leq \epsilon^{-2}\int_L|H|^2+ \epsilon^{-2}\Lambda \sup_L|\omega|^2 \int_L|\omega|^2.
\end{align}
Thus, as long as 
$$
\sup_L|\omega|^2 \leq \epsilon^2\frac{a}{4\Lambda}
$$ 
we arrive at 
\begin{align}
\frac{a}{4} \int_L|\omega|^2 &\leq\max\{1, \frac{\lambda_1^1}{\lambda^1_1-\lambda}\}\int_L|H|^2. \nonumber
\end{align}
\end{proof}

As a corollary we obtain the following result, which was obtained by Lotay-Pacini \cite[Theorem 5.2]{LP3}.

\begin{cor}\label{cor: LPpers}
Suppose $(X, \bar{g}_t, J_t)$ is a family of negative K\"ahler-Einstein structures and  $L\subset (X,\bar{g}_0, J_0)$ is a compact minimal Lagrangian submanifold.  Then for $|t|$ sufficiently small, there is a unique minimal Lagrangian $L_t$ near $L$.
\end{cor}
\begin{proof}
By a theorem of White \cite[Theorem 2.1]{White}, $L$ induces a smooth family of minimal submanifolds $L_t$ in $(X,\bar{g}_t, J_t)$. Since $L$ is Lagrangian, $L_t$ are totally real for $|t|$ small and $|\omega_t| < \epsilon$.  By Proposition~\ref{L2B} the submanifolds $L_t$ are Lagrangian.
\end{proof}

\section{Mean Curvature Flow}\label{sec: MCF}

The mean curvature flow evolves the immersion $F:L\rightarrow X$ by moving $L$ in the direction of its mean curvature vector. Specifically it is defined by the flow equation
\be
\label{MCF}
\frac d{dt}{F}=\vec{H}=-\eta^{mp}H_mN_p.
\ee
For notational convenience, we write $\hat H^p=-\eta^{mp}H_m$, so $\vec{H}=\hat H^pN_p$. Here $\{N_1,...,N_ n\}$ forms a basis for the normal bundle.  If we are using a general basis $\{e_\alpha\}$ for $TX$, we also choose, by a slight abuse of notation, to denote  the mean curvature vector by $\vec{H}=\hat H^\alpha e_\alpha$.

 In general many of our computations will take place on  the pullback  bundle $E:=F^* (TX)$ over $L$, which is time dependent (since $F$ is). Thus it is reasonable to consider $F$ as a smooth map from $L\times[0,T)$ to $ X\times [0,T)$, where we endow the latter with the metric $g+dt^2$. Under this map the push forward of the time vector field $\partial_t$ is $F_*(\partial_t) =  \vec{H}+\partial_t$.  If $(\vec V)=V^\alpha e_\alpha$ is a time dependent section of $E$, then the pullback of the metric connection is given by:
$$\overline{\nabla}_{\frac{d}{dt}} \vec V=\left(\frac{d}{dt} V^\alpha +\overline{\Gamma}_{\beta\gamma}^\alpha \hat H^\beta V^\gamma\right)e_\alpha.$$
Becuase the metric connection preserves the complex structure on $X$ we see the pulled back complex structure is parallel with respect to time.

For computations living in $TL$ and its associated bundles (such as the one form $H$ or the metric $g_{ij}$), the above covariant derivative reduces to the coordinate derivative $\frac{d}{dt}$, since we do not need to pull back the frame by $F$. In the case that $L$ is Lagrangian, the tangent bundle  is  identified with the normal bundle to $L$ via the complex structure. However, this is not true when $L$ is only assumed to be totally real, and the normal bundle will vary with time. Thus when we compute the evolution of normal quantities we need to use the covariant derivative in the time direction.

We  now turn to the evolution of various important quantities, beginning with  the tangent vector $F_i=F_*(\frac{\partial}{\partial x^i}).$ It will be important later on to keep track of the normal and tangental directions of many of these quantities, as we do in the following lemma.

\begin{lem}

The tangent vector $F_i$ evolves via the formula
\be
\label{Fievolve}
\overline{\nabla}_{\frac{d}{dt}}F_i=\hat H^p h_{pi}{}^s F_s+\hat\nabla_i \hat H^s N_s.
\ee
\end{lem}
\begin{proof}
Viewing $F_i$ as a vector in $E:=F^* (TX)$, we use the covariant expression for the time derivative:
\bea
\overline{\nabla}_{\frac{d}{dt}}F_i&=&\left(\frac{d}{dt} F_i^\alpha +\overline{\Gamma}_{\beta\gamma}^\alpha \hat H^\beta F_i^\gamma\right)e_\alpha\nonumber\\
&=&\left(\frac{d}{dx^i} \frac{d}{dt}F^\alpha +\overline{\Gamma}_{\beta\gamma}^\alpha \hat H^\beta F_i^\gamma\right)e_\alpha=\overline\nabla_{F_i}\vec H.\nonumber
\eea
At a point $p$ we can let $\{F_1,..,F_n, N_1,..., N_n\}$ be a basis for $TX$. This allows us to write
\bea
\overline{\nabla}_{F_i}\vec H&=&\langle \overline{\nabla}_{ F_i}\vec H, F_q\rangle F_s g^{qs}+\langle \overline{\nabla}_{F_i} \vec H, N_q\rangle N_s \eta^{qs}\nonumber\\
&=&-\langle \vec H,\overline{\nabla}_{ F_i}  F_q\rangle F_s g^{qs}+\pi_\perp(\overline\nabla_{F_i}\vec H).\nonumber
\eea
The first term on the right hand side is $-\hat H^p\langle N_p,\overline{\nabla}_{ F_i}  F_q\rangle F_s g^{qs}=\hat H^p h_{pi}{}^s F_s$. The second term on the right hand side is the covariant derivative on the normal bundle. As we have defined in Section \ref{normalsection}, this normal derivative can be expressed as   
$\pi_\perp(\overline\nabla_{F_i}\vec H)=(\hat\nabla_i \hat H^s)N_s $, where $\hat\nabla$ is defined in \eqref{hatconnection}. This gives the desired formula. 

\end {proof}

\begin{lem}
\label{metricflow}
Along the flow we have the following evolution equations:
\be
{\frac{d}{dt}}g_{ij}=2\hat H^ph_{p ij},\qquad
{\frac{d}{dt}}{\rm vol}_g=-\eta^{mp}H_mH_p{\rm vol}_g.\nonumber
\ee
\end{lem}

\begin{proof} We compute:
\bea 
 {\frac{d}{dt}}g_{ij}&=&\langle \overline{\nabla}_{\frac{d}{dt}}F_i,F_j\rangle+\langle F_i, \overline{\nabla}_{\frac{d}{dt}}F_j\rangle\nonumber\\
 &=&\langle \hat H^p h_{pi}{}^s F_s,F_j\rangle+\langle F_i,\hat H^p h_{pj}{}^m F_m\rangle=2\hat H^ph_{p ij}.\nonumber
\eea
Additionally
\be
\frac d{dt}{\rm vol}_g=\frac12g^{ij}\frac d{dt}g_{ij}{\rm vol}_g=g^{ij}\hat H^ph_{pij}{\rm vol}_g=\hat H^pH_p{\rm vol}_g=-\eta^{mp} H_mH_p {\rm vol}_g.\nonumber
\ee
\end{proof}
\begin{lem}
\label{omegaevolve}
The restricted K\"ahler form $\omega$ evolves via the equation $${\frac{d}{dt}} \omega =2dH.$$
\end{lem}
\begin{proof}
Using that $J$ is constant we see
\bea 
 {\frac{d}{dt}}\omega_{ij}&=&\langle \overline{\nabla}_{\frac{d}{dt}} JF_i,F_j\rangle+\langle J F_i, \overline{\nabla}_{\frac{d}{dt}}F_j\rangle\nonumber\\
 &=&\langle -\overline{\nabla}_{\frac{d}{dt}}F_i, JF_j\rangle+\langle J F_i, \overline{\nabla}_{\frac{d}{dt}}F_j\rangle\nonumber\\
 &=&\langle -\overline{\nabla}_{F_i}\vec H, JF_j\rangle+\langle\overline{\nabla}_{F_j}\vec H, JF_i\rangle.\nonumber
\eea
Meanwhile
$$\frac{\partial}{\partial x^i} H_j=-\frac{\partial}{\partial x^i}\langle \vec H,J F_j\rangle=-\langle\overline{\nabla}_{F_i}\vec H, JF_j\rangle-\langle\vec H, J(\overline{\nabla}_{F_i}F_j)\rangle.$$
Since $[\frac{\partial}{\partial x^i} ,\frac{\partial}{\partial x^j} ]=0$, it follows that $[F_i, F_j]=0$, and because $\overline\nabla$ is torsion free $\overline{\nabla}_{F_i}F_j=\overline{\nabla}_{F_j}F_i$. Thus
$$\frac{\partial}{\partial x^i} H_j-\frac{\partial}{\partial x^j} H_i=-\langle\overline{\nabla}_{F_i}\vec H, JF_j\rangle+\langle\overline{\nabla}_{F_j}\vec H, JF_i\rangle= {\frac{d}{dt}}\omega_{ij}.$$
\end{proof}

Because the time derivative of the restricted K\"ahler form $\omega$ is exact, we immediately get the following:
\begin{cor}
The cohomology class $[\omega ]\in H^2(L,\mathbb R)$ is fixed along the mean curvature flow.
\end{cor}
In light of the above result, we see $[\omega]=0$ is a necessary condition for the flow to deform $L$ to a minimal Lagrangian. As discussed in Section \ref{L2section}, the condition $[\omega]=0$ is an added assumption  in the Calabi-Yau setting, while when $\lambda\neq0$ this condition is automatic since  by Proposition \ref{dxiJ} we have $\omega=\frac1\lambda d\xi_J.$

Before we state our next result, we solidify our curvature convention. Let $\{e_\alpha\}$ be any basis for $T_pX$. 
We then set $\overline{R}_{\nu \rho \gamma\beta}:=\langle \overline{R}(e_\nu, e_\rho) e_\gamma, e_\beta\rangle$, which implies $[\overline{\nabla}_\nu,\overline{\nabla}_\rho] V^\beta=\overline{R}_{\nu \rho \gamma}{}^\beta V^\gamma$ for any section $V^\beta$ of $TX$.

\begin{lem}
Choose an arbitrary frame $e_\alpha$ for $T_pX$. The second fundamental form $A_{jk}^\alpha e_\alpha$ evolves via
\bea
 \overline\nabla_{\frac d{dt}}A_{ji}^\alpha&=&\left(\nabla_j(\hat H^p h_{pi}{}^s)  +\hat\nabla_i \hat H^s h_{sj}{}^s -\overline\nabla_{\frac d{dt}}\Gamma_{ji}^s\right)F_s^\alpha \nonumber\\
 &&+\hat\nabla_j(\hat\nabla_i \hat H^s) N_s^\alpha +\hat H^p h_{pi}{}^s (\overline\nabla_{F_j}F_{s})^\alpha+\overline{R}_{\nu\beta\gamma}{}^\alpha\hat H^\nu F_j^\beta F_i^\gamma.\nonumber
\eea
\end{lem}

We remark that this evolution equation shares two common terms with the evolution equation for $A_{ji}^\alpha$ from  Lemma 3.7 in \cite{Sm3}, namely the  curvature term $\overline{R}_{\nu\beta\gamma}{}^\alpha\hat H^\nu F_j^\beta F_i^\gamma$ and the term $ -\overline\nabla_{\frac d{dt}}\Gamma_{ji}^sF_s^\alpha$. The difference here is that the remaining term from  \cite{Sm3} is broken into tangental and normal components, which again is useful in the arguments to follow. We include a proof of this equation, keeping track of this decomposition, for completeness. 

\begin{proof}

We begin by looking at the evolution of two derivatives $F^\alpha$. Using \eqref{Fievolve} we have
\begin{align}
\frac{d}{dt}F_{ij}^\alpha =&\frac{\partial}{\partial x^j}\frac{d}{dt}F_i^\alpha=\frac{\partial}{\partial x^j}(\hat H^p h_{pi}{}^s F_s^\alpha+\hat\nabla_i \hat H^s N_s^\alpha-\overline{\Gamma}_{\beta\gamma}^\alpha \hat H^\beta F_i^\gamma)\nonumber\\
 =&\frac{\partial}{\partial x^j}(\hat H^p h_{pi}{}^s )F_s^\alpha+\hat H^p h_{pi}{}^s F_{js}^\alpha+\frac{\partial}{\partial x^j}(\hat\nabla_i \hat H^s)N_s^\alpha+\hat\nabla_i \hat H^s\frac{\partial}{\partial x^j}(N_s^\alpha)\nonumber\\
 &-\overline{\Gamma}_{\beta\gamma,\nu}^\alpha  \hat H^\beta F_i^\gamma F_j^\nu-\overline{\Gamma}_{\beta\gamma}^\alpha\frac{\partial}{\partial x^j} \hat H^\beta F_i^\gamma- \overline{\Gamma}_{\beta\gamma}^\alpha \hat H^\beta F_{ij}^\gamma.\nonumber
\end{align}
 The spacial derivative of the normal vector can be computed using Proposition \ref{normal} along with the formula for  the shape operator, which gives
\be
\frac{\partial}{\partial x^j}N_p^\alpha=h_{pj}{}^s F_s^\alpha+\hat\Gamma_{jp}^sN_s^\alpha -F_j^\kappa N_p^{\nu}\overline{\Gamma}_{\kappa\nu}^\alpha.\nonumber
\ee
Plugging in the above, and adding  $\overline{\Gamma}_{\beta\gamma}^\alpha \hat H^\beta F_{ij}^\gamma$ to each side, we see
 \begin{align}
\overline\nabla_{\frac d{dt}} F_{ij}^\alpha=&\frac{\partial}{\partial x^j}(\hat H^p h_{pi}{}^s )F_s^\alpha+\hat H^p h_{pi}{}^s F_{js}^\alpha+\left(\frac{\partial}{\partial x^j}(\hat\nabla_i \hat H^s)+\hat\Gamma_{jp}^s(\hat\nabla_i \hat H^p)\right)N_s^\alpha\nonumber\\
&+\hat\nabla_i \hat H^s h_{sj}{}^r F_r^\alpha -\hat\nabla_i \hat H^sF_j^\kappa N_s^{\nu}\overline{\Gamma}_{\kappa\nu}^\alpha-\overline{\Gamma}_{\beta\gamma,\nu}^\alpha  \hat H^\beta F_i^\gamma F_j^\nu\nonumber\\
&-\overline{\Gamma}_{\beta\gamma}^\alpha\frac{\partial}{\partial x^j} \hat H^\beta F_i^\gamma.\nonumber
\end{align}

We need to add the evolution of the connection term to the above evolution equation, which is computed as follows
\bea
\overline\nabla_{\frac d{dt}}(\overline\Gamma_{\beta\gamma}^\alpha F_j^\beta F_i^\gamma)&=&\overline\Gamma_{\beta\gamma,\nu}^\alpha \hat H^\nu F_j^\beta F_i^\gamma+\overline\Gamma_{\beta\gamma}^\alpha \left(\frac{d}{dt}F_j^\beta\right) F_i^\gamma\nonumber\\
&& +\overline\Gamma_{\beta\gamma}^\alpha F_j^\beta \left(\frac{d}{dt}F_i^\gamma\right)+\overline\Gamma_{\kappa\nu}^\alpha \overline\Gamma_{\beta\gamma}^\kappa \hat H^\nu F_j^\beta F_i^\gamma\nonumber\\
&=&\overline\Gamma_{\beta\gamma,\nu}^\alpha \hat H^\nu F_j^\beta F_i^\gamma+\overline\Gamma_{\beta\gamma}^\alpha \frac{\partial}{\partial x^j }\hat H^\beta F_i^\gamma+\overline\Gamma_{\kappa\nu}^\alpha \overline\Gamma_{\beta\gamma}^\kappa \hat H^\nu F_j^\beta F_i^\gamma\nonumber\\
&&+\overline\Gamma_{\beta\gamma}^\alpha F_j^\beta (\hat H^p h_{pi}{}^s F_s^\gamma+\hat\nabla_i \hat H^s N_s^\gamma-\overline{\Gamma}_{\nu\kappa}^\gamma \hat H^\nu F_i^\kappa),\nonumber
\eea
where in the final line above, the term in parenthesis comes from \eqref{Fievolve}. We can now add $\overline\nabla_{\frac d{dt}} F_{ij}^\alpha$ and  $\overline\nabla_{\frac d{dt}}(\overline\Gamma_{\beta\gamma}^\alpha F_j^\beta F_i^\gamma)$ to arrive at the evolution of $\overline\nabla_{F_j}F_{i}$, noting right away that the terms with $\overline{\Gamma}_{\beta\gamma}^\alpha\frac{\partial}{\partial x^j} \hat H^\beta F_i^\gamma$ cancel. In particular
\begin{align}
\left(\overline\nabla_{\frac d{dt}} \overline\nabla_{F_j}F_{i}\right)^\alpha=&\frac{\partial}{\partial x^j}(\hat H^p h_{pi}{}^s )F_s^\alpha+\hat H^p h_{pi}{}^s F_{js}^\alpha\nonumber\\
&+\left(\frac{\partial}{\partial x^j}(\hat\nabla_i \hat H^s)+\hat\Gamma_{jp}^s(\hat\nabla_i \hat H^p)\right)N_s^\alpha\nonumber\\
&+\hat\nabla_i \hat H^s h_{sj}{}^r F_r^\alpha -\hat\nabla_i \hat H^sF_j^\kappa N_s^{\nu}\overline{\Gamma}_{\kappa\nu}^\alpha-\overline{\Gamma}_{\beta\gamma,\nu}^\alpha  \hat H^\beta F_i^\gamma F_j^\nu\nonumber\\
&+\overline\Gamma_{\beta\gamma,\nu}^\alpha \hat H^\nu F_j^\beta F_i^\gamma +\overline\Gamma_{\kappa\nu}^\alpha \overline\Gamma_{\beta\gamma}^\kappa \hat H^\nu F_j^\beta F_i^\gamma\nonumber\\
&+\overline\Gamma_{\beta\gamma}^\alpha F_j^\beta (\hat H^p h_{pi}{}^s F_s^\gamma+\hat\nabla_i \hat H^s N_s^\gamma-\overline{\Gamma}_{\nu\kappa}^\gamma \hat H^\nu F_i^\kappa).\nonumber
\end{align}
The terms $ -\hat\nabla_i \hat H^sF_j^\kappa N_s^{\nu}\overline{\Gamma}_{\kappa\nu}^\alpha$ and $\overline\Gamma_{\beta\gamma}^\alpha F_j^\beta \hat\nabla_i \hat H^s N_s^\gamma$ cancel. Furthermore 
$$\left(\overline\Gamma_{\beta\gamma,\nu}^\alpha -\overline{\Gamma}_{\nu\gamma,\beta}^\alpha +\overline\Gamma_{\kappa\nu}^\alpha \overline\Gamma_{\beta\gamma}^\kappa  -\overline\Gamma_{\beta\kappa}^\alpha  \overline{\Gamma}_{\nu\gamma}^\kappa \right)\hat H^\nu F_j^\beta F_i^\gamma=\overline{R}_{\nu\beta\gamma}{}^\alpha\hat H^\nu F_j^\beta F_i^\gamma.$$
Finally, $\hat H^p h_{pi}{}^s F_{js}^\alpha+\overline\Gamma_{\beta\gamma}^\alpha F_j^\beta \hat H^p h_{pi}{}^s F_s^\gamma=\hat H^p h_{pi}{}^s (\overline\nabla_{F_j}F_{s})^\alpha.$ 
Thus we see
\bea
\left(\overline\nabla_{\frac d{dt}} \overline\nabla_{F_j}F_{i}\right)^\alpha&=&\frac{\partial}{\partial x^j}(\hat H^p h_{pi}{}^s )F_s^\alpha+\hat\nabla_i \hat H^s h_{sj}{}^r F_r^\alpha+\hat H^p h_{pi}{}^s (\overline\nabla_{F_j}F_{s})^\alpha\nonumber\\
 &&\left(\frac{\partial}{\partial x^j}(\hat\nabla_i \hat H^s)+\hat\Gamma_{jp}^s(\hat\nabla_i \hat H^p)\right)N_s^\alpha+\overline{R}_{\nu\beta\gamma}{}^\alpha\hat H^\nu F_j^\beta F_i^\gamma.\nonumber
\eea

Now, one can express the  projection of $\overline\nabla_{F_j}F_{i}$ onto the tangent bundle as a connection term, allowing us to write the second fundamental form as $A_{ji}^\alpha=\overline\nabla_{F_j}F_{i}^\alpha-\Gamma_{ji}^k F^\alpha_k.$ 
From the above evolution equation it follows that
\bea
 \overline\nabla_{\frac d{dt}}A_{ji}^\alpha&=&\frac{\partial}{\partial x^j}(\hat H^p h_{pi}{}^s )F_s^\alpha+\hat\nabla_i \hat H^s h_{sj}{}^r F_r^\alpha+\hat H^p h_{pi}{}^s (\overline\nabla_{F_j}F_{s})^\alpha\nonumber\\
 &&\left(\frac{\partial}{\partial x^j}(\hat\nabla_i \hat H^s)+\hat\Gamma_{jp}^s(\hat\nabla_i \hat H^p)\right)N_s^\alpha+\overline{R}_{\nu\beta\gamma}{}^\alpha\hat H^\nu F_j^\beta F_i^\gamma\nonumber\\
 &&-(\overline\nabla_{\frac d{dt}}\Gamma_{ji}^k)F_k^\alpha-\Gamma_{ji}^k\overline\nabla_{F_k} \hat H^\alpha.\nonumber
\eea
Breakling $\Gamma_{ji}^k\overline\nabla_{F_k} \hat H^\alpha$ into its tangential and normal directions gives
$$\Gamma_{ji}^k\overline\nabla_{F_k} \hat H^\alpha=\Gamma_{ji}^k\hat H^ph_{pk}{}^mF_m^\alpha+\Gamma_{ji}^k\hat\nabla_k \hat H^s N_s^\alpha.$$
Plugging this back in we see
\bea
 \overline\nabla_{\frac d{dt}}A_{ji}^\alpha&=&\left(\frac{\partial}{\partial x^j}(\hat H^p h_{pi}{}^s) -\Gamma_{ji}^k\hat H^ph_{pk}{}^s\right)F_s^\alpha+\hat\nabla_i \hat H^s h_{sj}{}^r F_r^\alpha\nonumber\\
 &&+\left(\frac{\partial}{\partial x^j}(\hat\nabla_i \hat H^s)+\hat\Gamma_{jp}^s(\hat\nabla_i \hat H^p)-\Gamma_{ji}^k\hat\nabla_k \hat H^s\right)N_s^\alpha\nonumber\\
 &&+\hat H^p h_{pi}{}^s (\overline\nabla_{F_j}F_{s})^\alpha+\overline{R}_{\nu\beta\gamma}{}^\alpha\hat H^\nu F_j^\beta F_i^\gamma-(\overline\nabla_{\frac d{dt}}\Gamma_{ji}^k)F_k^\alpha.\nonumber
\eea
This proves the desired formula.
\end{proof}

Using the result above,  we now  turn to the evolution of the individual components of the second fundamental from $h_{ijk}$.

\begin{lem}
Along the mean curvature flow $h_{ijk}$ evolves via
\bea
\frac d{dt}h_{ijk}&=&\hat H^p h_{pk}{}^sh_{ijs}-\eta_{si}\hat\nabla_j(\hat\nabla_k \hat H^s)-\hat H^p\langle \overline{R}(N_p, F_j)F_k,N_i\rangle\nonumber\\
 &&+h_{qjk}\hat H^\ell h_{\ell i}{}^q + h_{qjk}\omega^q{}_p \hat\nabla_i \hat H^p -h_{pjk}\omega_i{}^r\hat\nabla_r \hat H^p.\nonumber
\eea
\end{lem}
\begin{proof}
Recall $h_{ijk}=-\langle\overline\nabla_{F_j}F_k,N_i\rangle=-\langle A_{jk},N_i\rangle.$ Taking the time derivative yields
\bea
  \frac d{dt}h_{ijk}&=&-\langle  \overline\nabla_{\frac d{dt}}A_{jk},N_i\rangle-\langle A_{jk}, \overline\nabla_{\frac d{dt}}N_i\rangle.\nonumber
\eea
For simplicity we write down the two terms on the right hand side separately. First, we note
\begin{align}
-\langle    \overline\nabla_{\frac d{dt}}A_{jk},N_i\rangle=&-\eta_{si}\hat\nabla_j(\hat\nabla_k \hat H^s)-\hat H^p h_{pk}{}^s \langle \overline\nabla_{F_j}F_{s} ,N_i\rangle-\langle \overline{R}(\vec H, F_j)F_k,N_i\rangle\nonumber\\
=&-\eta_{si}\hat\nabla_j(\hat\nabla_k \hat H^s)-\hat H^p h_{pk}{}^sh_{ijs}-\hat H^p\langle \overline{R}(N_p, F_j)F_k,N_i\rangle.\nonumber
\end{align}
For the other term, since $A_{jk}=-\eta^{qs}h_{qjk}N_s$, we   have
\begin{align}
-\langle A_{jk}, \overline\nabla_{\frac d{dt}}N_i\rangle =&\eta^{qs}h_{qjk}\langle N_s, \overline\nabla_{\frac d{dt}}N_i\rangle=\eta^{qs}h_{qjk}\langle N_s, \overline\nabla_{\frac d{dt}}(J(F_i)-\omega_i{}^mF_m)\rangle\nonumber\\
=&\eta^{qs}h_{qjk}\langle N_s, J(\overline\nabla_{\frac d{dt}}F_i)\rangle -\eta^{qs}h_{qjk}\langle N_s,\omega _i{}^m\overline\nabla_{\frac d{dt}}F_m \rangle\nonumber\\
=&-\eta^{qs}h_{qjk}\langle J N_s, \overline\nabla_{F_i}\vec H\rangle -\eta^{qs}h_{qjk}\omega_i{}^m\langle N_s, \overline\nabla_{F_m}\vec H \rangle\nonumber
\end{align}
Note that
\begin{align}
-\eta^{qs}h_{qjk}\langle J N_s, \overline\nabla_{F_i}\vec H\rangle =&\eta^{qs}h_{qjk}\langle \eta_{sm}g^{mr}F_r, \overline\nabla_{F_i}\vec H\rangle+\eta^{qs}h_{qjk}\langle\omega_s{}^rN_r, \overline\nabla_{F_i}\vec H\rangle\nonumber\\
=&-g^{qr}h_{qjk}\langle \overline\nabla_{F_i}F_r, \vec H\rangle+\eta^{qs}h_{qjk}\langle\omega_s{}^rN_r, \overline\nabla_{F_i}\vec H\rangle, \nonumber
\end{align}
from which it follows that
\be
-\langle A_{jk}, \overline\nabla_{\frac d{dt}}N_i\rangle =g^{qr}h_{qjk}\hat H^\ell h_{\ell ir} +\eta^{qs}h_{qjk}\omega_s{}^r\eta_{rp}\hat\nabla_i \hat H^p -\eta^{qs}h_{qjk}\omega_i{}^m\eta_{sp}\hat\nabla_m \hat H^p.\nonumber
\ee
 Putting everything together we see
 \bea
 \overline\nabla_{\frac d{dt}}h_{ijk}&=&\hat H^p h_{pk}{}^sh_{ijs}-\eta_{si}\hat\nabla_j(\hat\nabla_k \hat H^s)-\hat H^p\langle \overline{R}(N_p, F_j)F_k,N_i\rangle\nonumber\\
 &&+h_{qjk}\hat H^\ell h_{\ell i}{}^q +\eta^{qs}h_{qjk}\omega_s{}^r\eta_{rp}\hat\nabla_i \hat H^p -h_{pjk}\omega_i{}^r\hat\nabla_r \hat H^p.\nonumber
\eea

Finally, we see $\omega_{s}{}^r\eta_{rp}=\omega^r{}_p\eta_{sr}$, since in normal coordinates
$$
\omega_{s}{}^r\eta_{rp}=\omega_{sq}g^{rq}(g_{rp}+\omega_{r\ell}g^{\ell m}\omega_{mp})=\omega_{sp}+\omega_{sq}\omega_{q\ell}\omega_{\ell p},
$$
while
$$
\omega^r{}_p\eta_{sr}=\omega_{ip}g^{ri}(g_{sr}+\omega_{sj}g^{j k}\omega_{kr})=\omega_{sp}+\omega_{ip}\omega_{s j}\omega_{ji}.
$$
This completes the proof of the formula.
\end{proof}

At last, we turn to the evolution equation for the mean curvature 1-form $H_i=g^{jk}h_{ijk}$. First, recall that $\hat\nabla\eta_{\ell m}=0$, which implies
 $$-\eta_{si}\hat\nabla_j(\hat\nabla_k \hat H^s)=\eta_{si}\hat\nabla_j(\hat\nabla_k \eta^{sp} H_p)=\hat\nabla_j(\hat\nabla_k  H_i).$$
Now, from the above Lemma we can directly compute
\bea
 \frac d{dt}H_i&=&-g^{jp}\frac d{dt}g_{pq}g^{qk}h_{ijk}+g^{jk} \frac d{dt}h_{ijk}\nonumber\\
 &=&-2\hat H^ph_{p}{}^{jk}h_{ijk}+g^{jk}\hat H^p h_{pk}{}^sh_{ijs}+ g^{jk}\hat\nabla_j(\hat\nabla_k H_i)+H_q\hat H^\ell h_{\ell i}{}^q\nonumber\\
 && + H_q\omega^q{}_p \hat\nabla_i \hat H^p -H_q\omega_i{}^r\hat\nabla_r \hat H^p-g^{jk}\hat H^p\langle \overline{R}(N_p, F_j)F_k,N_i\rangle.\nonumber
\eea
Combining and rearranging terms, and using the definition of $\hat H^p$,  gives
\begin{align}
\label{Hevolve1}
 \overline\nabla_{\frac d{dt}}H_i=&  g^{jk}\hat\nabla_j(\hat\nabla_k H_i) +\eta^{ps} H_s h_{p}{}^{jk}h_{ijk}-\eta^{\ell s}H_q H_s h_{\ell i}{}^q\nonumber\\
 & + H_q\omega^q{}_p \hat\nabla_i \hat H^p -H_q\omega_i{}^r\hat\nabla_r \hat H^p-g^{jk}\hat H^p\langle \overline{R}(N_p, F_j)F_k,N_i\rangle.
\end{align}
While this evolution equation will be useful to us later on, there are two reasons we will also need a different formula for the evolution of $H$. First, the terms $\eta^{ps} H_s h_{p}{}^{jk}h_{ijk}$ and $-\eta^{\ell s}H_q H_s h_{\ell i}{}^q$ do not cancel, and in particular this gives a term of order $H*h^2*\eta^{-1}$, which is   difficult  to keep track of. More importantly the curvature term is not useful as is, because it is not a Ricci curvature term.

To see how we can extract a Ricci curvature term from the evolution of $H$, we first need the following lemma:
\begin{lem} Contracting the curvature term $\langle \overline{R}(N_p, F_j)F_k,N_i\rangle$ with $g^{ik}$ produces a Ricci curvature term. In particular we have the formula
   \begin{align}
   \label{Thing101}
\overline{\rm Ric}( F_p, F_j)=&g^{ik}\langle \overline{R}( N_p , F_j) F_k, N_i\rangle+\frac12  \omega^{ki}\langle \overline{R}(JF_p, F_j)  F_k ,  F_i\rangle \nonumber\\
  &+\frac12\eta^{ik}\omega_k{}^r\langle \overline{R}( JF_p , F_j)   N_r, N_i\rangle+ g^{ik}\omega_p{}^r\langle \overline{R}( F_r , F_j) F_k, N_i\rangle.
   \end{align}
\end{lem}
\begin{proof}
For any basis $\{e_\alpha\}$ of $T_pX$ we see $\{Je_\alpha\}$ also forms a basis, and thus the Ricci curvature  can be expressed as $\overline{\rm Ric}(F_p, F_j)= \overline g^{\alpha\beta}\langle \overline{R}(F_p,  Je_{\alpha})J e_{\beta} , F_j\rangle.$ As a result we see
 \ba
 \overline{\rm Ric}(F_p, F_j)&=& -\overline g^{\alpha\beta}\langle \overline{R}(F_p, J e_{\alpha}) e_{\beta} , JF_j\rangle\nonumber\\
 &=&\overline g^{\alpha\beta}\langle \overline{R}(J e_{\alpha}, e_{\beta}) F_p, JF_j\rangle+\overline g^{\alpha\beta}\langle \overline{R}(e_{\beta},F_p) J e_{\alpha} , JF_j\rangle\nonumber\\
 &=&\overline g^{\alpha\beta}\langle \overline{R}(J e_{\alpha}, e_{\beta}) F_p , JF_j\rangle-\overline g^{\alpha\beta}\langle \overline{R}(F_p,e_\beta)  e_{\alpha} , F_j\rangle.\nonumber
 \ea
 The second term on the right is again the Ricci curvature. This implies that 
$$ 2  \overline{\rm Ric}(F_p, F_j)=\overline g^{\alpha\beta}\langle \overline{R}(J e_{\alpha}, e_{\beta}) F_p , JF_j\rangle.$$

With the  basis $F_1,...,F_n, N_1,...,N_n$ we have
\be
2  \overline{\rm Ric}(F_p, F_j)=g^{ik}\langle \overline{R}(J F_i, F_k) F_p , JF_j\rangle+\eta^{ik}\langle \overline{R}(J N_i, N_k) F_p , JF_j\rangle.\nonumber
\ee
Using symmetries of the curvature tensor along with the fact that $J$ is covariant constant gives
\ba
2  \overline{\rm Ric}(F_p, F_j)&=-g^{ik}\langle \overline{R}(J F_i, F_k)  JF_p , F_j\rangle-\eta^{ik}\langle \overline{R}(J N_i, N_k) JF_p , F_j\rangle\nonumber\\
&=-g^{ik}\langle \overline{R}(J F_p, F_j) J F_i ,  F_k\rangle-\eta^{ik}\langle \overline{R}(JF_p, F_j)  J N_i , N_k\rangle\nonumber\\
&=g^{ik}\langle \overline{R}(J F_p, F_j) F_k ,  JF_i\rangle-\eta^{ik}\langle \overline{R}(JF_p, F_j)  J N_k , N_i\rangle.\nonumber
\ea
Using \eqref{Jtangent}, first term on the right hand side above can be expressed as
$$g^{ik}\langle \overline{R}(J F_p, F_j) F_k ,  JF_i\rangle=g^{ik}\langle \overline{R}(J F_p, F_j) F_k ,  N_i\rangle+ \omega^{ki}\langle \overline{R}(J F_p, F_j) F_k ,  F_i\rangle,$$
while with an application of \eqref{Jnormal} we can write the second term as
\ba-\eta^{ik}\langle \overline{R}(JF_p, F_j)  J N_k , N_i\rangle&=&g^{ik}\langle \overline{R}(JF_p, F_j)  F_k , N_i\rangle\nonumber\\
&&+\eta^{ik}\omega_{k}{}^r\langle \overline{R}(JF_p, F_j)  N_r , N_i\rangle.\nonumber
\ea
Putting everything together gives
\ba
2  \overline{\rm Ric}(F_p, F_j)&=&2g^{ik}\langle \overline{R}(J F_p, F_j) F_k ,  N_i\rangle+\omega^{ki}\langle \overline{R}(J F_p, F_j) F_k ,  F_i\rangle\nonumber\\
&&+\eta^{ik}\omega_{k}{}^r\langle \overline{R}(JF_p, F_j)  N_r , N_i\rangle.\nonumber
\ea
A final application of   \eqref{Jtangent} to $J F_p$ inside of the term $2g^{ik}\langle \overline{R}(J F_p, F_j) F_k ,  N_i\rangle$ completes the proof of the formula.
 \end{proof}

As a result, if we want to arrive at a Ricci curvature term, we need to contract the time derivative of $h_{ijk}$ with $g^{ik}$ instead of $g^{jk}.$ To accomplish this we use the identity $H=\xi-d^*\omega$. We remark that in some literature on the Lagrangian mean curvature flow   (see for example  \cite{Sm2}), the contraction $g^{ik}h_{ijk}$ is also used to compute the time derivate of $H_j$, since in the Lagrangian case there is full symmetry on the indices of $h_{ijk}$.

 Returning to the time derivative of $H_j$ we now compute
 \begin{align}
 \frac d{dt} H_j  =&  \frac d{dt}  (g^{ik}h_{jik}) =  \frac d{dt} g^{ik}(h_{ijk}+\nabla_k\omega_{ij})+g^{ik} \frac d{dt}(h_{ijk}+\nabla_k\omega_{ij}).\nonumber\\
=&-2\hat H^p h_p{}^{ik}h_{ijk}-2\hat H^p h_p{}^{ik}\nabla_k\omega_{ij}+g^{ik}\hat\nabla_j(\hat\nabla_k   H_i)\nonumber\\
&-g^{ik}\hat H^p\langle \overline{R}(N_p, F_j)F_k,N_i\rangle + \hat H^p h_{p}{}^{is}h_{ijs}+h_{qj}{}^i\hat H^\ell h_{\ell i}{}^q \nonumber\\
&+ h_{qj}{}^i\omega^q{}_p \hat\nabla_i \hat H^p -h_{pj}{}^i\omega_i{}^r\hat\nabla_r \hat H^p+g^{ik} \frac d{dt} \nabla_k\omega_{ij}.\nonumber
   \end{align}
Note that now  $\hat H^p h_{p}{}^{is}h_{ijs}$ is equal to $h_{qj}{}^i\hat H^\ell h_{\ell i}{}^q$, and these terms cancel with $-2\hat H^p h_p{}^{ik}h_{ikk}$. Thus
 \bea
 \frac d{dt}  H_j  &=& g^{ik}\hat\nabla_j(\hat\nabla_k   H_i)-g^{ik}\hat H^p\langle \overline{R}(N_p, F_j)F_k,N_i\rangle-2\hat H^p h_p{}^{ik}\nabla_k\omega_{ij}\nonumber\\
 && + h_{qj}{}^i\omega^q{}_p \hat\nabla_i \hat H^p -h_{pj}{}^i\omega_i{}^r\hat\nabla_r \hat H^p+g^{ik} \frac d{dt}\nabla_k\omega_{ij}.\nonumber
\eea
The term $g^{ik}\hat H^p\langle \overline{R}(N_p, F_j)F_k,N_i\rangle$ now contains a Ricci curvature term, which will be necessary  for achieving exponential decay.

We expand on this formula further by turning to the evolution of $ \frac d{dt} \nabla_k\omega_{ij}$.
\bea
g^{ik} \frac d{dt}\nabla_k\omega_{ij}=-g^{ik} \frac d{dt} \Gamma_{ki}^m\omega_{mj}-g^{ik} \frac d{dt}\Gamma_{kj}^m\omega_{im}-2(d^*d H)_j,\nonumber
\eea
where we have used $\frac{d}{dt}\omega=2dH$.
Now, using the well known formula
$$\frac{d}{dt}\Gamma_{ki}^m=\frac12g^{m\ell}\left(\nabla_k(\frac{d}{dt}g_{i\ell })+\nabla_i(\frac{d}{dt}g_{k\ell})-\nabla_\ell(\frac{d}{dt}g_{ik})\right),$$
our evolution equation for $g$ gives
$$\frac{d}{dt}\Gamma_{ki}^m=g^{m\ell}\left(\nabla_k(\hat H^p h_{pi\ell })+\nabla_i(\hat H^ph_{pk\ell})-\nabla_\ell(\hat H^p h_{pik})\right).$$
Plugging this in we at last arrive at
\begin{align} 
 \label{Hevolve2}
 \frac d{dt} H_j =& g^{ik}\hat\nabla_j(\hat\nabla_k   H_i)-g^{ik}\hat H^p\langle \overline{R}(N_p, F_j)F_k,N_i\rangle-2\hat H^p h_p{}^{ik}\nabla_k\omega_{ij} \\
  &+h_{qj}{}^i\omega^q{}_p \hat\nabla_i \hat H^p -h_{pj}{}^i\omega_i{}^r\hat\nabla_r \hat H^p-g^{ik} \nabla_k(\hat H^ph_{pi\ell})\omega^\ell{}_{j} \nonumber\\
  &-g^{ik} \nabla_i(\hat H^ph_{pk\ell})\omega^{\ell}{}_j+g^{ik} \nabla_\ell(\hat H^ph_{pik})\omega^\ell{}_j-g^{ik} \nabla_k(\hat H^ph_{pj\ell})\omega_{i}{}^\ell\nonumber\\
&-g^{ik} \nabla_j(\hat H^ph_{pk\ell})\omega_{i}{}^\ell+g^{ik} \nabla_\ell(\hat H^ph_{pjk })\omega_{i}{}^\ell -2(d^{*_g}d H)_j.\nonumber
\end{align}
This formula will play a fundamental role in the subsequent section.

We   conclude this section by  stating  higher order second fundamental form bounds. These   bounds hold for any embedded submanifold of a Riemannian manifold (of any codimension) evolving along the mean curvature flow (see Theorem 3.2 from \cite{CY}, or Proposition A.2 in \cite{CJL}). However, here we state them for our specific geometric setup.
\begin{prop}
\label{smoothing}
 Let $L_t$ be a totally real submanifold of $(X,\overline g, \overline \omega)$ evolving smoothly via the mean curvature flow starting at $L_0$. Suppose there exists a constant $\Lambda$ so that for all time $t\in[0,\frac{\alpha}{\Lambda})$, and all $0\leq \ell \leq m+1$ we have $\sup_{L_t}|A|^2\leq\Lambda$, and $\sup_{L_t}|\nabla^\ell{\overline {Rm}}|^2\leq \Lambda^{2+\ell}$. Then there exists  a constant $C>0$ depending on $\alpha$, $n$,  $k$, and $m$ so that
  $$\sup_{L_t}|\nabla^m A|^2\leq\frac{C \Lambda}{t^m}.$$

 \end{prop}

 \section{Exponential decay}\label{sec: expDecay}

In this section we demonstrate exponential decay of the $L^2$ norm of the mean curvature $1$-form $H$, for a family of totally real submanifolds $L_t$ evolving along the mean curvature flow. Unless explicitly specified, we use the convention that if the norm of a tensor is not adorned, then it is understood to be taken with respect to the metric induced by $\overline g$. For example,  $|H|$ denotes the norm taken with respect to $g$ on $\Omega^1(L_t)$. The mean curvature vector $\vec H$ is a section of the normal bundle, so $|\vec H|$ denotes the norm with respect to the induced metric $\eta$. If $|\omega|$ is sufficiently small then $g$ and $\eta$ are equivalent, i.e.
\be
\label{equivmetrics}
 \frac{1}{1+\sup|\omega|^2} |H|^2_g \leq |H|^2_\eta\leq \frac{1}{1-\sup|\omega|^2}  |H|^2_g. 
\ee
Furthermore, in our work  above  $\langle\cdot,\cdot\rangle$ denotes the inner product with respect to $\overline g$, so  to avoid confusion  with this fact we will write    $\langle\cdot,\cdot\rangle_g$   for the inner product with respect to $g$. Additionally, as above,  $\lambda$ denotes the K\"ahler-Einstein constant and $\lambda_1^1(t)$ the smallest non-zero eigenvalue of the Hodge Laplacian on $\Omega^1(L_t)$.

\begin{prop} 
\label{exponential}
Let $L_t$ be a totally real submanifold of $(X,\overline g, \overline \omega)$ evolving via the mean curvature flow for $t\in(0, T]$, with $T\leq\infty$. Assume  $C^0$ control of the second fundamental form $\sup_{L_t}|A|^2\leq\Lambda$ for all $t\in(0,T]$. Additionally assume that $L_t$ lies inside of $B_{\Lambda^{-\frac12}R}(L_0)$, and that in this ball $\overline g$ satisfies the local curvature bounds $|\overline{Rm}|\leq\Lambda$. Fix ${t_0}>0$ and consider the time interval $({t_0}, T]$.  If $\lambda<0$ let $a=-\lambda$, and if $\lambda\geq 0$ let
$$
a:=\inf_{t\in(t_0, T] } \lambda^{1}_1(t) -\lambda >0.
$$
 If $(X,\bar{g})$ is Calabi-Yau, assume in addition that $[\omega]|_{L}=0 \in H^{2}(L, \mathbb{R})$. Finally, for notationally simplicity set $\Psi: =\sqrt{\frac{\Lambda}{t_0}}+\Lambda$. Under the above assumptions there exists a constant $C>1$, depending only on $n$, so that if for times $t\in({t_0}, T]$ it holds 
\begin{equation}
\label{exponentialassumption}
\begin{aligned}
\Lambda^{-\frac12} \sup_{L_t}|H|+\sup_{L_t}|\omega| &<\frac{1}{C}\min\left\{1, \frac{a}{\Lambda}\right\}\min\left\{1, \frac{a^{1/2}}{\Psi^{1/2}}, \frac{a^2}{\Lambda\Psi}\right\}\min\left\{1, \left(\frac{\lambda^{1}_{1}-\lambda}{\lambda^{1}_1}\right)^{\frac{1}{2}}\right\}\\
\sup_{L_t}|\omega|^2 &< \bigg| \frac{a}{(2\lambda+a)}\bigg| 
\end{aligned}
\end{equation}
then we have the exponential decay $$ \int_{L_t}  |H|^2  \leq   2e^{-\frac{a t}{20}}\int_{L_0} |H|^2. $$

\end{prop}


\begin{proof}
We drop the subscript from $L_t$ for notational simplicity, and unless  otherwise specified we use the convention all terms unadorned terms are defined at the time $t$. Using \eqref{Hevolve2}, we have  the evolution equation \begin{align} 
\label{evolveH2}
 {\frac d{dt}} \frac12\int_L|H|^2 =& \int_L\left( \frac12{\frac d{dt}}g^{jq}H_jH_q+g^{jq}{\frac d{dt}}H_jH_q\right)  \\
 =&\int_L\Big(-  \hat H^ph_{p}{}^{jq}H_jH_q+g^{jq}g^{ik} \hat\nabla_j(\hat\nabla_k   H_i)H_q\nonumber\\
 &-g^{jq}g^{ik}\hat H^p\langle \overline{R}(N_p, F_j)F_k,N_i\rangle H_q\nonumber\\
 &+ h_{m}{}^{qi}\omega^m{}_p \hat\nabla_i \hat H^pH_q -h_{p}{}^{qi}\omega_i{}^r\hat\nabla_r \hat H^pH_q\nonumber\\
 &  -\nabla_k(\hat H^ph_{p}{}^{km})\omega_{m}{}^qH_q- \nabla_i(\hat H^ph_{p}{}^{im})\omega_{m}{}^qH_q \nonumber\\
 &+ \nabla_\ell(\hat H^pH_p)\omega^{\ell q}H_q-   \nabla_k(\hat H^ph_{p}{}^{qm})\omega^k{}_{m}H_q\nonumber\\
 &-g^{jq} \nabla_j(\hat H^ph_{p}{}^{im})\omega_{im}H_q+  \nabla_\ell(\hat H^ph_{p}{}^{qi})\omega_{i}{}^\ell H_q\nonumber\\
 &-2g^{jq}\hat H^p h_p{}^{ik}\nabla_k\omega_{ij}H_q-2g^{jq}(d^*d H)_jH_q\Big).\nonumber
\end{align}
The main idea of the proof is to show that one can find a small negative term involving $\int_L|H|^2$ on the right hand side above. Then, as long as $\sup_L|H|$ and $\sup_L|\omega|$ are chosen sufficiently  small,  the negative term will dominate. In practice, and due to some challenging cross terms, we will also have to modify the quantity by adding terms involving $|\omega|^2$.

The last term on the right hand side of \eqref{evolveH2} is a useful negative term, given by
$$-2\int_L\langle d^*dH, H\rangle_g=-2\int_L\langle dH, dH\rangle_g=-2||dH||_{L^2}^2.$$ 
The first term on the right hand side is controlled by $ 2\sqrt\Lambda \sup_L|H| \int_L|H|^2.$
For terms of the form    $ \nabla(\hat H* h)*\omega*H$,  we can distribute the derivative inside of the parenthesis. If the derivative lands on  $h $, because  $|\nabla h|^2$ is controlled by $C\Lambda/{t_0}$ by Proposition \ref{smoothing} (where $C$ only depends on $n$), we have
$$\int_L |\hat H*\nabla( h)*\omega*H| \leq \frac{4C\sqrt{\Lambda}}{\sqrt{t_0}}\sup_L|\omega|\int_L|H|^2.$$
The above bound will also hold for the term $ \nabla_\ell(\hat H^pH_p)\omega^{\ell q}H_q$. Thus we arrive at
\begin{align} 
\label{thing100}
 {\frac d{dt}} \frac12\int_L|H|^2 \leq&\int_L\Big(  g^{jq}g^{ik} \hat\nabla_j(\hat\nabla_k   H_i)H_q -g^{jq}g^{ik}\hat H^p\langle \overline{R}(N_p, F_j)F_k,N_i\rangle H_q \\
 &-2g^{jq}\hat H^p h_p{}^{ik}\nabla_k\omega_{ij}H_q+ h_{m}{}^{qi}\omega^m{}_p \hat\nabla_i \hat H^pH_q\nonumber\\
 & -h_{p}{}^{qi}\omega_i{}^r\hat\nabla_r \hat H^pH_q- \nabla_k \hat H^ph_{p}{}^{km} \omega_{m}{}^qH_q  \nonumber\\
 &  - \nabla_i \hat H^ph_{p}{}^{im} \omega_{m}{}^qH_q-   \nabla_k \hat H^ph_{p}{}^{qm}\omega^k{}_{m}H_q \nonumber\\
 &-g^{jq} \nabla_j\hat H^ph_{p}{}^{im}\omega_{im}H_q+  \nabla_\ell\hat H^ph_{p}{}^{qi}\omega_{i}{}^\ell H_q\Big)\nonumber\\
 &-2\int_L|dH|^2+C\left( \sqrt{\Lambda} \sup_L|H|+ \sqrt{\frac{\Lambda}{t_0}}\sup_L|\omega|\right)\int_L|H|^2.\nonumber
\end{align}
Again $C$ is a constant that depends only on $n$, which may change from line to line.

We would like to integrate by parts in the first term on the right hand side above to get a good negative term. However, because $\hat\nabla$ does not preserve the the metric $g$ or the volume form, we first must change this to $\nabla$. Using equation \eqref{derivchange} gives
\be
g^{ik}\hat\nabla_j(\hat\nabla_k   H_i)=g^{ik}\nabla_j(\hat\nabla_k   H_i)- \eta^{\ell p}\omega^{kr} h_{pjr} \hat\nabla_k   H_\ell + \eta^{\ell p}\omega_{p}{}^rh_{rj}{}^k \hat\nabla_k   H_\ell. \nonumber
\ee
After multiplying by $g^{jq}H_q$ the  last two terms on the right hand side   are of the form $\eta^{-1}*h*\omega*H*\hat\nabla \hat H$. Applying \eqref{derivchange}  once more 
$$|\eta^{-1}*h*\omega*H*\hat\nabla H|\leq 2\sqrt{\Lambda} |\omega||H||\hat\nabla H|\leq2\sqrt{\Lambda} |\omega||H|| \nabla H|+4\Lambda |\omega|^2|H|^2.$$
Next, we see there are 7  terms of the form $h*\omega*H* \nabla \hat H$ from \eqref{thing100}. When the derivative $\nabla$ lands on $\eta^{-1}$ from $\hat H$, it gives an extra $2\nabla\omega*\omega$, and $\nabla\omega$ is controlled by $2\sqrt\Lambda$ by Proposition \ref{derivativeo}. Thus
\be
\label{thing4001}
|h*\omega*H* \nabla \hat H|\leq 2\sqrt{\Lambda} |\omega| |H| |\nabla   H|+4 \Lambda|\omega|^2|H|^2.
\ee
Returning to \eqref{thing100},  we see
\begin{align} 
 {\frac d{dt}} \frac12\int_L|H|^2 \leq&\int_L\Big(  g^{jq}g^{ik}  \nabla_j(\hat\nabla_k   H_i)H_q -g^{jq}g^{ik}\hat H^p\langle \overline{R}(N_p, F_j)F_k,N_i\rangle H_q\nonumber \\
 &-2g^{jq}\hat H^p h_p{}^{ik}\nabla_k\omega_{ij}H_q+16\sqrt\Lambda |\omega||H||\nabla H|\Big)-2\int_L|dH|^2 \nonumber\\
 &+C\left(  \sqrt{\Lambda}\sup_L|H|+  \left(\sqrt{\frac{\Lambda}{t_0}}+\Lambda\right)\sup_L|\omega|\right)\int_L|H|^2.\nonumber
\end{align}
Note we have used that   $|\omega|^2\leq |\omega|$ for $|\omega|$ small.

Consider the term containing $2g^{jq}\hat H^p h_p{}^{ik}\nabla_k\omega_{ij}H_q$. Integrating by parts, when  the derivative lands on $h$, we only contribute to the constant $C$. When the derivative lands on $H$, we contribute $4$ to $\int_L\sqrt\Lambda |\omega||H||\nabla H|$. Finally, when the derivative lands on $\hat H$, we can appeal to \eqref{thing4001}, and contribute both to $\int_L\sqrt\Lambda |\omega||H||\nabla H|$ and the constant $C$.

 Next we turn to the curvature term. By \eqref{Thing101} we see
  \begin{align}
-g^{jq}g^{ik}\hat H^p \langle \overline{R}( N_p , F_j) F_k, N_i\rangle H_q=&g^{jq} \eta^{ps}H_s H_q \overline{\rm Ric}( F_p, F_j)\nonumber\\
&+\frac12 g^{jq} \eta^{ps}H_s H_q \omega^{ik}\langle \overline{R}(JF_p, F_j)  F_k ,  F_i\rangle\nonumber\\
  &-\frac12g^{jq} \eta^{ps}H_s H_q\eta^{ik}\omega_k{}^r\langle \overline{R}( JF_p , F_j)   N_r, N_i\rangle\nonumber\\
  &- g^{jq} \eta^{ps}H_s H_qg^{ik}\omega_p{}^r\langle \overline{R}( F_r , F_j) F_k, N_i\rangle .\nonumber
\end{align}
 For $|\omega|$ small enough, all the non-Ricci curvature terms are controlled by $6 |\omega||\overline{Rm}||H|^2$. Using that $\overline{\rm Ric}( F_p, F_j)=\lambda g_{pj}$, the first term above is equal to
 $$\lambda g^{jq} \eta^{ps}H_s H_qg_{pj}= \lambda \eta^{qs}H_s H_q=\lambda |H|^2_\eta,  $$
Thus we have extracted a useful term from the Ricci curvature.
  
Next, integrating by parts on the $g^{jq}g^{ik}  \nabla_j(\hat\nabla_k   H_i)H_q$ term now gives
   \bea
\int_L g^{jq}g^{ik}  \nabla_j(\hat\nabla_k   H_i)H_q&=&-\int_L g^{jq}g^{ik} \hat\nabla_k   H_i\nabla_jH_q\nonumber\\
&\leq&-\int_L|d^*H|^2+2\sqrt\Lambda \int_L|\omega||H||\nabla H|.\nonumber
\eea
 As stated above, for notational simplicity we let $\Psi: =\sqrt{\frac{\Lambda}{t_0}}+\Lambda$. We have assumed that near $L_t$ the curvature tensor of $\overline g$ satisfies $|\overline{Rm}|\leq\Lambda$. Putting everything together so far we arrive at
\begin{align} 
\label{thingawesome}
 {\frac d{dt}} \frac12\int_L|H|^2 \leq&-\int_L \left(|d^*H|^2+2|dH|^2\right)+26\sqrt\Lambda\int_L   |\omega||H||\nabla H|  \\
& +\lambda\int_{L}|H|^2_{\eta} + \left(C\sqrt{\Lambda}\sup_L|H|+ C\Psi\sup_L|\omega|\right)\int_L|H|^2.\nonumber
\end{align}

The main difficulty we have at this point is the $ |\omega||H||\nabla H|$ term. In order to control it,  instead of just looking at the time derivative of $||H||_{L^2}^2,$ we add  $B|\omega|^2|H|^2$ to the integral, where $B$ is a large constant to be determined, and compute the time derivative 
$${\frac d{dt}} \int_L\left(\frac12+ B|\omega|^2 \right)|H|^2.$$
Recall the evolution equation $\frac{d}{dt}\omega=2dH$ from Lemma \ref{omegaevolve}. Combining this with Proposition \ref{formuladH} we see that
\be
\label{thing309}
\frac{d}{dt}\frac12|\omega|^2\leq \langle\Delta\omega,\omega\rangle_g+\lambda|\omega|^2+10\left(\sup_X|\overline{Rm}|+\Lambda\right)|\omega|^2
\ee
(where we used the Gauss-Codazzi equation to control the curvature  terms $Rm$ by  $\overline{Rm}$ and $\Lambda$). Note $\lambda$ is a fixed constant controlled by the curvature bound.  Thus
\be
\label{thing310}
\frac{d}{dt} |\omega|^2\leq 2\langle\Delta\omega,\omega\rangle_g+20\Psi |\omega|^2.
\ee

Recall  \eqref{Hevolve1}, which gives the evolution  of $H$  without worrying about creating a Ricci curvature term
\bea
 {\frac d{dt}}H_i&=&  g^{jk}\hat\nabla_j(\hat\nabla_k H_i) +\eta^{ps} H_s h_{p}{}^{jk}h_{ijk}-\eta^{\ell s}H_q H_s h_{\ell i}{}^q\nonumber\\
 && + H_q\omega^q{}_p \hat\nabla_i \hat H^p -H_q\omega_i{}^r\hat\nabla_r \hat H^p-g^{jk}\hat H^p\langle \overline{R}(N_p, F_j)F_k,N_i\rangle.\nonumber
\eea
From here we compute
\begin{align}
 {\frac d{dt}} |H|^2=& -2\hat H^ph_{p}{}^{im} H_i H_m+ 2g^{im}\frac{d}{dt} H_i H_m\nonumber\\
 =& 2g^{im}g^{jk}\hat\nabla_j(\hat\nabla_k H_i)H_m -2\hat H^ph_{p}{}^{im} H_i H_m +2g^{im}\eta^{ps} H_s h_{p}{}^{jk}h_{ijk}H_m\nonumber\\
 & -2g^{im}\eta^{\ell s}H_q H_s h_{\ell i}{}^qH_m+ 2g^{im}H_q\omega^q{}_p \hat\nabla_i \hat H^pH_m \nonumber\\
 &-2g^{im}H_q\omega_i{}^r\hat\nabla_r \hat H^pH_m-2g^{im}g^{jk}\hat H^p\langle \overline{R}(N_p, F_j)F_k,N_i\rangle H_m.\nonumber
\end{align}
Using that $|h|^2$ is controlled by $\Lambda$, $|\nabla H|\leq C\sqrt{\Lambda/t_0}$, and $|\omega|\leq 1$,  we see 
\bea
 {\frac d{dt}}|H|^2&\leq& 2g^{im}g^{jk}\hat\nabla_j(\hat\nabla_k H_i)H_m +C\Psi|H|^2.\nonumber
\eea
Here, as above, $C$ only depends on $n$.

We can now write down the evolution equation:
\begin{align}
 {\frac d{dt}}\frac12\int_L |\omega|^2 |H|^2 \leq& \int_L\Big( \langle\Delta\omega,\omega\rangle_g|H|^2+C\Psi |\omega|^2|H|^2 \nonumber\\
 &+ g^{im}g^{jk}\hat\nabla_j(\hat\nabla_k H_i)H_m  |\omega|^2\Big) .\nonumber
\end{align}
Before we integrate by parts, note that changing $\hat\nabla_j$ to $\nabla_j$ creates a term containing   $\omega*h$ which is bounded by $\sqrt\Lambda|\omega|^3|H||\hat\nabla H|$. Also we can  change the resulting $|\hat\nabla H|$  to $|\nabla H|$ at the cost of adding to the $C\Psi|\omega|^2|H|^2$ term. This gives
\begin{align}
\label{thing11}
 {\frac d{dt}}\frac12\int_L |\omega|^2 |H|^2 \leq& \int_L\Big( \langle\Delta\omega,\omega\rangle_g|H|^2+C\Psi |\omega|^2|H|^2 \\
 &+  g^{im}g^{jk} \nabla_j(\hat\nabla_k H_i)H_m  |\omega|^2+\sqrt\Lambda|\omega|^3|H|| \nabla H|\Big) .\nonumber
\end{align}
Integrating by parts on the first term on the right hand side above we see
\begin{align}
\int_L \langle\Delta\omega,\omega\rangle_g|H|^2=&\int_L\Big( - |\nabla\omega|^2|H|^2-2g^{jk}\langle\nabla_j\omega,\omega\rangle_g\langle \nabla_k H, H\rangle_g \Big)\nonumber\\
& \leq \int_L2|\nabla\omega||\omega||\nabla H||H|. \nonumber
\end{align}
Note we do not keep the negative term. Next,  integrating by parts on  Laplacian term for $H$ from \eqref{thing11} we see 
\begin{align}
\int_L g^{im}g^{jk}\nabla_j(\hat\nabla_k H_i)H_m |\omega|^2\leq -\int_L \left( \langle \hat\nabla H,\nabla H\rangle_g   |\omega|^2 +2 |\hat\nabla  H| | H| |\nabla\omega||\omega|\right).\nonumber
\end{align}
We want to remove the hat on the derivative of $H$ on the right hand side above. Doing so gives
\begin{align}
\int_L g^{im}g^{jk}\nabla_j(\hat\nabla_k H_i)H_m |\omega|^2\leq&\int_L\Big(- |\nabla H|^2 |\omega|^2 +2 |\nabla  H| | H| |\nabla\omega||\omega|\nonumber\\
&+\sqrt{\Lambda}|H||\nabla H||\omega|^3+2\sqrt{\Lambda}|H|^2|\nabla\omega||\omega|^2\Big).\nonumber
\end{align}
Note the last term on the right hand side above is controlled by $C\Psi |\omega|^2|H|^2$, since $|\nabla \omega|\leq 2|h|$.

Returning to  \eqref{thing11} and  pugging in what we have computed  gives
\begin{align}
\label{thing55}
  {\frac d{dt}}\frac12\int_L |\omega|^2 |H|^2 \leq&\int_L\Big(- |\nabla H|^2 |\omega|^2+4 |\nabla  H| | H| |\nabla\omega||\omega| \\
&+2\sqrt{\Lambda}|H||\nabla H||\omega|^3+C\Psi |\omega|^2|H|^2\Big).\nonumber
\end{align}
Applying Young's inequality to second term on the right hand side above
$$4 |\nabla  H| | H| |\nabla\omega||\omega|\leq\frac14|\nabla H|^2|\omega|^2+16|H|^2|\nabla \omega|^2.$$
Furthermore, we can also apply Young's inequality to the third term from the right hand side of \eqref{thing55} to see
$$2\sqrt{\Lambda}|H||\nabla H||\omega|^3\leq \frac14|\nabla H|^2|\omega|^2+4\Lambda |\omega|^4|H|^2.$$
Note since $|\omega|\leq 1$ the second term on the right is controlled by $C\Psi|\omega| |H|^2$. Thus  \eqref{thing55} becomes 
\begin{align}
  {\frac d{dt}}\frac12\int_L  |\omega|^2 |H|^2 &\leq\int_L\Big(-\frac12 |\nabla H|^2 |\omega|^2+16 | H|^2 |\nabla\omega|^2 +C\Psi |\omega| |H|^2\Big).\nonumber
\end{align}

We are ready to add the above inequality (multiplied by   $2B$  for a constant $B\geq1$ to be determined) to inequality \eqref{thingawesome}, yielding
\begin{align} 
\label{thingawesome2}
 {\frac d{dt}} \int_L\left(\frac12+B|\omega|^2\right)|H|^2 \leq&\int_L \left(-|d^*H|^2-2|dH|^2-B |\nabla H|^2 |\omega|^2\right)\\
 &+\int_L  \left(26\sqrt\Lambda|\omega||H||\nabla H|+32B | H|^2 |\nabla\omega|^2\right)\nonumber  \\
& +\lambda\int_{L}|H|^2_{\eta}+\left(C\sqrt{\Lambda}\sup_L|H|+ CB\Psi\sup_L|\omega|\right)\int_L|H|^2.\nonumber
\end{align}
The  $- |d^*H|^2$ and $-|dH|^2$ terms can be used to extract a helpful negative term.  A standard spectral theory argument yields
$$
\int_{L}|d^*H|^2 + |dH|^2=\int_L\langle \Box H, H\rangle_g \geq \lambda_1^1(t) \int|H|^2.
$$
Plugging this back into  \eqref{thingawesome2} and using \eqref{equivmetrics},  together with the assumption of the proposition, we arrive at
\begin{align} 
 {\frac d{dt}} \int_L\left(\frac12+B|\omega|^2\right)|H|^2 \leq&\left(-\frac{a}{2} + C\sqrt{\Lambda}\sup_L|H|+ CB\Psi\sup_L|\omega|\right)\int_L|H|^2.\nonumber\\
 &+\left(32B\sup_L|H|^2\right)\int_L |\nabla\omega|^2\nonumber\\
 &+\int_L  \left(26\sqrt\Lambda|\omega||H||\nabla H|-B |\nabla H|^2 |\omega|^2\right).\nonumber 
\end{align}

To control the term involving  $26\sqrt\Lambda|\omega||H||\nabla H|$ we first write this expression as follows:
\begin{align*} 
26\sqrt\Lambda|\omega||H||\nabla H| &=\left(\sqrt{\frac{ a }{10}}|H|\right)\left(\sqrt{\frac{10}{ a} }26\sqrt\Lambda|\omega||\nabla H|\right)\\
&\leq \frac{ a}{20}|H|^2+\frac{3380\Lambda}{a}|\omega|^2|\nabla H|^2.
\end{align*} 
We can now specify the large constant  $B=\frac{3380\Lambda}{ a} $, which is   a scale invariant quantity with explicit dependence on bounded terms. The second term on the right above cancels with the term containing $-B |\nabla H|^2 |\omega|^2$, yielding
\begin{align} 
 {\frac d{dt}} \int_L\left(\frac12+B|\omega|^2\right)|H|^2 \leq&\left(-\frac{3  a}{20}+C\sqrt{\Lambda}\sup_L|H|+ CB\Psi\sup_L|\omega|\right)\int_L|H|^2\nonumber\\
 &+\left( 20\Lambda  \sup_L|\omega|^2 +32B\sup_L|H|^2\right)\int_L |\nabla\omega|^2.\nonumber
 \end{align}

Next we need to control the terms involving $\int_L |\nabla\omega|^2$. This is accomplished  by adding a $q|\omega|^2$ term to the time derivative, where  $q$ is a small constant to be determined. Recall  \eqref{thing310},   which implies
\begin{align} 
\frac{d}{dt} \int_L q|\omega|^2&\leq -2q\int_L|\nabla\omega|^2+20q\Psi\int_L |\omega|^2\nonumber\\
&\leq  -2q\int_L|\nabla\omega|^2+ \frac{80q\Psi}{a}\max\left\{1, \frac{\lambda_1^1}{\lambda_1-\lambda}\right\}\int_L |H|^2,\nonumber
\end{align}
where we used Proposition~\ref{L2B}.  If we choose 
$$
q<\min\left\{\frac{a^2}{1600\Psi},\frac{a}{16}\right\}\cdot \min\left\{1, \frac{\lambda_1^1-\lambda}{\lambda^1_1}\right\},
$$
then the right most term above is bounded by $\frac{ a}{20}\int_L|H|^2$. 

Putting everything together we now arrive at
\begin{align} 
 {\frac d{dt}} \int_L\Big(q|\omega|^2 &+\frac12|H|^2+B|\omega|^2|H|^2\Big) \nonumber\\
\leq& \left(-\frac{a}{10}+C\sqrt{\Lambda}\sup_L|H|+ CB\Psi\sup_L|\omega|\right)\int_L|H|^2\nonumber\\
 &+ \left(-2q+20 \Lambda   \sup_L|\omega|^2 +32B\sup_L|H|^2\right)\int_L |\nabla\omega|^2.\nonumber
 \end{align}
To complete the argument, we need 
 $$
 \begin{aligned}
 \sup_L|H|&<\min\left\{\sqrt{\frac{q}{32B}},\frac{ a}{40C\sqrt{\Lambda}}\right\},\\
 \sup_L|\omega|&<\min\left\{\sqrt{ {\frac{q}{20\Lambda }}},\frac{ a}{40CB\Psi}, \frac{1}{2\sqrt{B}}\right\}
 \end{aligned}
 $$
 These conditions can be written succinctly as
 $$
 \begin{aligned}
 \Lambda^{-\frac{1}{2}}\sup_{L}|H| &\leq\frac{a}{C\Lambda}\ \min\{1, \frac{a^{1/2}}{\Psi^{1/2}}\} \cdot\left(\min\left\{1, \frac{\lambda_1^1-\lambda}{\lambda^1_1}\right\}\right)^{1/2}\\
 \sup_{L}|\omega| &\leq \frac{1}{C} \min\{1, \frac{a^{1/2}}{\Psi^{1/2}},  \frac{a^2}{\Lambda \Psi}\}\cdot\left(\min\left\{1, \frac{\lambda_1^1-\lambda}{\lambda^1_1}\right\}\right)^{1/2},
 \end{aligned}
 $$
 each of which is implied by the main assumption \eqref{exponentialassumption}.  This  now gives
  \begin{align}
  {\frac d{dt}} \int_L\left(q|\omega|^2+\frac12|H|^2+B|\omega|^2 |H|^2\right)&\leq - \frac{ a}{20}  \int_L  |H|^2.\nonumber
 \end{align}

Finally, we  make use of two simple observations. First, Proposition~\ref{L2B} implies 
\be
\label{line1}
-\frac14 \int_L|H|^2\leq -\frac{a}{16}\min\left\{1, \frac{\lambda_1^1-\lambda}{\lambda^1_1}\right\}\int_L|\omega|^2\leq -q\int_L|\omega|^2,
\ee
 where the last inequality follows from the definition of $q$. Second,  we assumed $\sup_L|\omega|^2<\frac1{4B},$ yielding
\be
 \label{line2}
 -\frac14 \int_L|H|^2\leq -B\sup_L|\omega|^2 \int_L|H|^2\leq -\int_LB|\omega|^2 |H|^2. 
 \ee
 Taken together we now conclude 
  \begin{align}
   - \frac{ a}{20}  \int_L  |H|^2&= - \frac{\  a}{20}  \int_L  \left(\frac12 |H|^2+\frac14|H|^2+\frac14|H|^2\right)\nonumber\\
   &\leq - \frac{ a}{20}  \int_L  \left(\frac12 |H|^2+q|\omega|^2+B|\omega|^2 |H|^2\right),\nonumber
    \end{align}
giving exponential decay of $ \int_L  \left(\frac12 |H|^2+q|\omega|^2+B|\omega|^2 |H|^2\right)$. The exponential decay of the $L^2$ norm of $H$ follows, as demonstrated by
  \begin{align}\frac12\int_{L_t}  |H|^2 &\leq\int_{L_t} \left(q|\omega|^2+\frac12|H|^2+B|\omega|^2 |H|^2\right)\nonumber\\
  &\leq e^{-\frac{  a t}{20}}\int_{L_0}\left(q|\omega|^2+\frac12|H|^2+B|\omega|^2 |H|^2\right)\nonumber\\
  &\leq  e^{-\frac{  a t}{20}}\int_{L_0}|H|^2.\nonumber
      \end{align}
In the last line we used the negative of \eqref{line1} and \eqref{line2} at time $t=0$. This completes the proof of the proposition.

    \end{proof}

    \section{Control of further geometric quantities}\label{sec: geomQuant}

   In the previous section we saw that the smallest nonzero eigenvalue  $\lambda^1_1$ of the Hodge  Laplacian on $\Omega^1(L)$ plays an important role in proving the exponential decay of $ \int_L|H|^2$ along the mean curvature flow. In our convergence result to follow, we will also need control of the non-collapsing constant $\kappa$. In this section we show how these two   geometric quantities evolve  along the mean curvature flow and how they can be controlled.

     \begin{defn}
We say $(L,g)$ is $\kappa$-non collapsed at scale $r_0$ if for every $0<r<r_0$ and for every $p\in L$ we have
$${\rm Vol}_g(B_r(p))\geq \kappa r^{n},$$
where $B_r(p)$ denotes the geodesic ball of radius $r$ about $p$.
    \end{defn}

 We now state \cite[Lemma 3.5]{Li}, which we modify slightly here in order to achieve scale invariant estimates.
            \begin{lem}
        \label{noncollapsed}
 Suppose that $(L,g)$ is $\kappa$-noncollapsed at scale $r_0$. Let $S$ be a tensor on $L$ such that
$ |\nabla S|\leq C_0$ and $\int_L|S|^2\leq m.$
We then have the following bound
$$\sup_L|S|\leq \max\left\{\frac{2 C_0^{\frac{n}{n+2}}}{\kappa^{\frac 1{n+2}}}, \frac{2^{\frac{n+2}2}m^{\frac{n}{2(n+2)}}}{\sqrt{\kappa r_0^n}}\right\}m^{\frac1{n+2}}.$$
 \end{lem}     
    
    \begin{proof}
    Suppose $|S|$ achieves a maximum at a point $p\in L$. Consider the ball $  B_r(p)$, with $r=\min\{\frac{r_0}2,\frac{|S(p)|}{2C_0}\}$. By the gradient bound for $S$, at any point   $x\in B_r(p)$ we have     $|S(x)|\geq \frac12|S(p)|.$
Thus 
    $$m\geq\int_{B_r(p)}|S|^2\geq\frac14|S(p)|^2{\rm Vol}_g(B_r(p))\geq\frac14|S(p)|^2\kappa r^n.$$
 If $r_0\leq \frac {|S(p)|}{C_0}$,  plugging in $r=\frac{r_0}2$  gives
 $$|S(p)|\leq \sqrt{\frac{2^{n+2}m}{\kappa r_0^n}}.$$
 Otherwise plugging in $r=\frac{|S(p)|}{2C_0}$ we see
  $$|S(p)|\leq 2 \left(\frac{C_0^n m}{\kappa  }\right)^{\frac 1{n+2}}.$$
   Taking the maximum of the two quantities on the right hand side and factoring out $m^{\frac 1{n+2}}$ completes the proof.    \end{proof}

To take advantage of the above lemma, we need  the evolving metric  $g_t$ to stay $\kappa$-noncollapsed  along the mean curvature flow. First, define
 \be
 \label{mu}
 \mu(t):=\int_0^t2\sup_{L_s}(|A||H|) ds.
 \ee
We now demonstrate the following:
\begin{lem}
\label{kappaflow}
  Suppose $L_0$ is $\kappa $-noncollapsed at scale $r_0$ at $t=0$. For later times $L_t$ is $\kappa e^{-(n+1)\mu(t)}$-noncollapsed at scale $r_0.$
   \end{lem}     
 
 \begin{proof}
By Lemma \ref{metricflow}, we see
$$|{\frac{d}{dt}}g_{ij}|\leq 2|A||H|\qquad{\rm and}\qquad |{\frac{d}{dt}}{\rm vol}_g|\leq 2|H|^2{\rm vol}_g ,$$
from which we conclude the distance function $d_{g_t}(p,q)$ on $L$ satisfies
$$e^{-\mu(t)}d_{g_0}(p,q)\leq d_{g_t}(p,q)\leq e^{\mu(t)}d_{g_0}(p,q)$$
as well as
$$ e^{-\mu(t)}{\rm vol}_{g_0}\leq {\rm vol}_{g_t}\leq e^{\mu(t)}{\rm vol}_{g_0}.$$
 Putting these together gives, for $r\leq r_0$,
 $${\rm Vol}_{g_t}(B_r(p))=\int_{B_r(p)}{\rm vol}_{g_t}\geq \int_{B_{e^{-\mu(t)}r}(p)}e^{-\mu(t)}{\rm vol}_{g_0}\geq \kappa e^{-(n+1)\mu(t)}r^n.$$
   \end{proof}
Thus we see that as long as the second fundamental form and mean curvature stay controlled in $C^0$ along the flow, so will the non-collapsing constant.

We now turn to the control along the mean curvature flow of $\lambda^1_1$.  Note that if $\lambda^0_1$ denotes the first non-zero eigenvalue of the scalar Laplacian, and $\rho_1$ denotes the smallest nonzero eigenvalue of the Hodge Laplacian on $d^*$-exact one forms, then by Hodge decomposition $\lambda^1_1=\min\{\lambda_1^0,\rho_1\}$, and so it suffices to estimate each quantity individually.  In \cite[Lemma 2.3]{CJL} it is demonstrated that along the mean curvature flow
\begin{equation}
\label{1999}
 e^{-3\mu(t)}\lambda_1^0(0)\leq \lambda_1^0(t)\leq e^{3\mu(t)}\lambda_1^0(0).
 \end{equation}
 Thus, to demonstrate that $\lambda^1_1$ does not degenerate along the flow, we need control of $\rho_1(t)$. In fact, using a similar Rayleigh quotient argument as the one from \cite{CJL}, we can prove:
    \begin{lem}
Let $\rho_1(t)$ be the first eigenvalue of the Hodge Laplacian $\Box:W^{2,2}(\Omega^1(L)) \cap {\rm ker}(d)^\perp\rightarrow L^2(\Omega^1(L))\cap {\rm ker}(d)^\perp$ with respect to $g_t$ evolving along the mean curvature flow. Then one has the following control
  \begin{align}
  	    e^{-2\mu(t)} \rho_1( 0)	\leq \rho_1( t)\leq e^{2\mu(t)} \rho_1( 0).\nonumber
  \end{align}

  \end{lem}
 \begin{proof} The space ${\rm ker}(d)^\perp$ is isomorphic to the range of $d^*$ on $\Omega^2(L)$.  Therefore, for $\beta\in{\rm ker}(d)^\perp$, we have $\Box\beta=d^*d\beta$ and $(\Box \beta,\beta)_{L^2}=|d\beta|^2_{L^2}$. Thus, $\rho_1(t)$ can be computed by the Rayleigh quotient 
 	  \begin{align*}
 	  	  \rho_1(t)=\inf_{\beta\in{\rm ker}(d)^\perp} \frac{\int_{L} |d\beta|_{g_t}^2\mbox{vol}_{g_t}}{\int_L |\beta|^2_{g_t}\mbox{vol}_{g_t}}.
 	  \end{align*} 
	  
Consider the orthogonal projection $\pi_t:  L^2(\Omega^1(L))\rightarrow  L^2(\Omega^1(L))\cap {\rm ker}(d)^\perp.$ For any $1$-form $\beta$ we have $d\beta=d\pi_t\beta$ since $(I-\pi_t)\beta$ is $d$-closed. This allows us to rewrite  the above Rayleigh quotient as 
   \begin{align} \label{ray}
   	  \rho_t(t)=\inf_{\beta\notin{ \rm ker}(d) } \frac{\int_{L } |d\beta|_{g_t}^2\mbox{vol}_{g_t}}{\int_L |\pi_t\beta|^2_{g_t}\mbox{vol}_{g_t}}.
     \end{align} 
     The advantage of the second expression is that while ${\rm ker}(d)^\perp$ depends on the metric $g_t$, the space ${ \rm ker}(d)$ is constant along the flow.  Thus, the $1$-from $\beta$ will not change along the  mean curvature flow.

   We now separately estimate the evolution along the flow of the denominator and the numerator of the Rayleigh quotient. First   \begin{align*}
     	\bigg|\frac{d}{dt}\int_L|d\beta|^2_{g_t}\mbox{vol}_{g_t}\bigg|\leq  2\|\dot{g}_t \|_{L^{\infty}(g_t)} \int_L |d\beta|^2_{g_t}\mbox{vol}_{g_t}.
     \end{align*} 
     Applying   Lemma \ref{metricflow} gives  $|{\frac{d}{dt}}g_{ij}|\leq 2|A||H|$, which implies
     \begin{align} \label{1011}
    e^{-\mu(t)} \int_L|d\beta|^2_{g_0}\mbox{vol}_{g_0}	\leq \int_L|d\beta|^2_{g_t}\mbox{vol}_{g_t}\leq e^{\mu(t)} \int_L|d\beta|^2_{g_0}\mbox{vol}_{g_0}.
    \end{align}
   On the other hand, observe that the denominator of \eqref{ray} can be rewritten as 
     \begin{align*}
     	\int_L|\pi_t \beta|^2_{g_t}\mbox{vol}_{g_t}=\inf_{\gamma\in {\rm ker}(d)} \int_L |\beta-\gamma|^2_{g_t}\mbox{vol}_{g_t}.
     \end{align*}
     Again from the evolution of the metric, we have 
    \begin{align*}
    		\bigg|\frac{d}{dt}\int_L|\beta-\gamma|^2_{g_t}\mbox{vol}_{g_t}\bigg|\leq  2\|\dot{g}_t \|_{L^{\infty}(g_t)}\int_L |\beta-\gamma|^2_{g_t}\mbox{vol}_{g_t},
    \end{align*} which implies that 
  \begin{align*}
  	   e^{-\mu(t)} \int_L|\beta-\gamma|^2_{g_0}\mbox{vol}_{g_0}	\leq \int_L|\beta-\gamma|^2_{g_t}\mbox{vol}_{g_t}\leq e^{\mu(t)} \int_L|\beta-\gamma|^2_{g_t}\mbox{vol}_{g_0}.
  \end{align*} 
  Taking infimum on both sides over $\gamma\in {\rm ker}(d)$, we reach 
  \begin{align}\label{1012}
  	  e^{-\mu(t)} \int_L|\pi_0 \beta |^2_{g_0}\mbox{vol}_{g_0}	\leq \int_L|\pi_t\beta|^2_{g_t}\mbox{vol}_{g_t}\leq e^{\mu(t)} \int_L|\pi_0\beta|^2_{g_0}\mbox{vol}_{g_0}.
 \end{align}	
   The lemma then follows by combining \eqref{1011} and \eqref{1012}.
 \end{proof} 

The above lemma, along with \eqref{1999}, allows us to conclude 
\begin{equation}
\label{1998}
 e^{-3\mu(t)}\lambda_1^1(0)\leq \lambda_1^1(t)\leq e^{3\mu(t)}\lambda_1^1(0),
 \end{equation}
 giving the desired control. 

  \section{Convergence}\label{sec: convergence}

We are now ready to prove our main convergence result. Recall that if $|\omega|$ is sufficiently small then $g$ and $\eta$ are equivalent, i.e. \eqref{equivmetrics} holds. As a result if $|H|$ is exponentially decreasing, so is $|\vec H|$. Since certain formulas along the mean curvature flow are easier to state using $\vec H$, this will be an important observation. Additionally, for this section we focus on the $\lambda\geq0$ case, as the negative K\"ahler-Einstein case  is done in the exact same way.

Fix a totally real submanifold $L_0$ of $(X,\overline{g},\overline{\omega})$. Assume $L_0$ has the following control of geometric quantities: the second fundamental form satisfies 
  $\sup_{L_0}|A|^2\leq\Lambda$, the volume is given by ${\rm Vol}(L_0)= V$, the first nonzero eigenvalues $\lambda_1^1$ associated to the Hodge Laplacian on $\Omega^1(L_0)$ satisfies $\lambda_1^1-\lambda \geq a$, and   $L_0$ is $\kappa$-noncollapsed on scale $r$. Finally, assume that the mean curvature one form $H$ and restricted K\"ahler form $\omega$ satisfy the estimate
  $$\sup_{L_0}\Lambda ^{-1}|H|^2+ \sup_{L_0}|\omega|^2\leq\epsilon.$$
Let $L_t$ be the solution of the mean curvature flow starting at $L_0$.  

    \begin{defn}
  \label{geometry}
 We say the solution to the mean curvature flow $L_t$ has  $(b, \epsilon)$-control at time $t$ if
  \begin{enumerate}
  \item The second fundamental form satisfies
  $$|A(t)|^2\leq b\Lambda.$$
  \item The volume satisfies
  $$( b V)^{-1}\leq{\rm Vol}(L_t)\leq  b V.$$
  \item The first nonzero eigenvalue $\lambda_1^1(t)$  associated to the Hodge Laplacian on  $\Omega^1(L)$ satisfies
    $$ \lambda_1^1(t)-\lambda \geq a b^{-1}.$$ 
 \item $L_t$ is $\kappa b^{-1}$-noncollapsed on scale $r$.
  \item The mean curvature one form $H$ and the restricted K\"ahler form $\omega$ satisfy
 $$\sup_{L}\Lambda^{-1}|H|^2+ \sup_{L}|\omega|^2\leq\epsilon.$$
  \end{enumerate}
   \end{defn}
   
   We also assume that
   \begin{equation}
   \label{eq: LocalCurvatureControl}
\sup_{x \in B_{\Lambda^{-1/2}R}(L_0)}|\nabla^{\ell}\overline{Rm}|^2 \leq \Lambda^{2+\ell} \quad \text{ for } 0\leq \ell \leq 5.
\end{equation} 
Our first goal is to achieve short time control of the above  quantities. 
\begin{lem}
  \label{interval}
Suppose $L_0$ has $(1,\epsilon)$-control and let $L_t$ be the solution to the mean curvature flow starting at $L_0$.  Suppose that~\eqref{eq: LocalCurvatureControl} holds. Then, for all $\delta\leq 1$, there exists a constant $C$ depending only on $n$ such that, if we set
$$
t_{\delta} =\frac{1}{C\Lambda} \min\left\{\delta, \epsilon^{-1/2}R,  \delta \epsilon^{-1/2}\min\left\{1,\frac{\lambda^{1}_1(0) -\lambda}{\lambda^1_1(0)}\right\}\right\},
$$ 
then for all $t\in[0, t_{\delta}]$,  $L_t$ has $((1+\delta), (1+\delta)5\epsilon)$ control, and $L_t \subseteq B_{\frac12\Lambda^{-{1}/{2}}R}(L_0)$.
\end{lem}
\begin{proof}
Define times
$$
\begin{aligned}
T_A&:=\sup\{s>0\,:\, |A|^2(t)\leq (1+\delta)\Lambda\,\text{ for all } t\in[0,s)\}\\
T_\omega&:=\sup\{s>0\,:\, |\omega|^2(t)\leq (1+\delta)\epsilon\,\text{ for all } t\in[0,s)\}\\
T_H&:=\sup\{s>0\,:\, |H|^2(t)\leq (1+\delta)4\Lambda \epsilon\,\text{ for all } t\in[0,s)\}\\
T_{R}&:= \sup \{s>0\,:\, L_{t} \subset B_{\frac12\Lambda^{-1/2}R}(L_0)\, \text{ for all } t\in[0,s)\}
\end{aligned}
$$
By integrating the mean curvature we have $T_R \geq \min\{T_{H}, \frac{1}{C\Lambda}\epsilon^{-\frac{1}{2}}R\}$.

We first estimate $T_{A}$.  Suppose that $T_{A} \leq T_{R}$.  By  \cite[Corollary 3.9]{Sm3}, along the mean curvature flow the second fundamental form squared satisfies the differential inequality:
 \begin{equation}\label{eq: 2ndFundEvo}
  \frac{d}{dt} |A|^2\leq \Delta |A|^2+4|A|^4+20|\overline {Rm}| |A|^2+ 4|\nabla\overline {Rm}||A|.
  \end{equation}
We show a lower bound for $T_A$ that depends only on $\Lambda$, $\delta$. Note that since we are assuming $T_A \leq T_{R}$  it holds
   $$ \frac{d}{dt} |A|^2\leq \Delta |A|^2+200\Lambda^2$$
for $t\in[0, T_{A})$. By the maximum principle
      \begin{align}
      \label{Abound}
    |A|^2(t)&\leq \sup_{L_0}|A|^2(0)+200\Lambda^2t.
    \end{align}
    Thus,  $T_{A} \geq \frac{\delta}{C\Lambda}$ for a fixed constant $C$.  In conclusion we have that, in any case, $T_{A} \geq \min\{T_{R}, \frac{\delta}{C\Lambda}\}$.
 
We next consider the restricted K\"ahler form.  If we assume that $T_{\omega} \leq \min\{T_A, T_{R}\}$, then inequality \eqref{thing309}  implies
    \begin{align}
\frac{d}{dt}|\omega|^2 &\leq \Delta|\omega|^2+100\Lambda |\omega|^2.
    \end{align}
By the maximum principle and ODE comparison we  conclude $T_{\omega} \geq \frac{\delta}{C \Lambda}$ for a fixed constant $C$.  Thus,
$$
T_{\omega} \geq \min\{T_{A}, T_{R}, \frac{\delta}{C\Lambda}\}\geq \min\{T_{R}, \frac{\delta}{C\Lambda}\}.
$$

We turn to the mean curvature.  It suffices to estimate $T_{H}$ under the assumption that $T_{H} \leq T_{R}$. It is convenient to use \cite[Corollary 3.8]{Sm3}, which gives the evolution of the mean curvature vector along the mean curvature flow in any codimension: 
   \begin{align}
   \label{mcfvectorevolution}
    \frac{d}{dt} |\vec H|^2\leq \Delta |\vec H|^2+4|A|^2|\vec H|^2+2|\overline{Rm}| |\vec H|^2.
    \end{align}
 As long as $t$ lies in the interval $[0,\min\{T_A,T_\omega\})$,   we can choose $\epsilon$ smaller than a dimensional constant so that \eqref{equivmetrics} applies. In this case we have  $|\vec{H}|^2(t)\leq 16\Lambda\epsilon$,   and  so
     \begin{align}
    \frac{d}{dt} |\vec H|^2 &\leq \Delta |\vec H|^2+10\Lambda|\vec H|^2.\nonumber
    \end{align}
    By the maximum principle and ODE comparison we conclude that $T_{H} \geq \frac{\delta}{C\Lambda}$ for a fixed constant $C$.  Thus, we have that if $T_{H} \leq \min\{T_{R}, T_{A}, T_{\omega}\}$, then
    $$
    T_{H} \geq  \frac{\delta}{C\Lambda}.
    $$
As a result we conclude that
$$
\min\{T_A,T_\omega, T_{H}, T_{R}\} \geq \frac{1}{C\Lambda} \min\{\delta, \epsilon^{-\frac{1}{2}}R\}.
$$

Finally, recall the quantity $\mu(t)$ from \eqref{mu}. We have already seen in the proof of Lemma \ref{kappaflow} that the volume form satisfies $ e^{-\mu(t)}{\rm vol}_{g_0}\leq {\rm vol}_{g_t}\leq e^{\mu(t)}{\rm vol}_{g_0}.$ Using \eqref{Abound} and the definition of $\mu(t)$ it is easy to conclude  $(1+\delta)^{-1}V^{-1}\leq {\rm vol}_{g_t} \leq (1+\delta) V$ for $t$ small enough.  The non-collapsing constant $\kappa$ satisies the desired bounds as well.

Equation  \eqref{1998} gives $\lambda_1^1(t)\geq e^{-3\mu(t)}\lambda_1^1(0)$. Thus, as long as
\be
\label{equation99}
\mu(t)\leq-\frac13{\rm log}\left(\frac{\lambda_1^1(0)+\delta\lambda}{(1+\delta)\lambda_1^1(0)}\right)
\ee
we see
\be
\label{equation100}
\lambda_1^1(t)-\lambda\geq e^{-3\mu(t)}\lambda_1^1(0)-\lambda\geq \frac{\lambda_1^1(0)+\delta\lambda}{(1+\delta)}-\lambda\geq \frac{\lambda_1^1(0)- \lambda}{(1+\delta)}\geq\frac{a}{1+\delta}.
\ee
In fact, the expression on the right hand side of \eqref{equation99} is bounded below by $\frac{\delta}{1+\delta}\frac{\lambda_1^1(0)-\lambda}{3\lambda_1^1(0)}$, and as a result we only need $\mu(t)$ less than this quantity. By the definition of $\mu(t)$ it then suffices to choose $t$ so that
$$
t \leq \frac{\delta}{\Lambda C} \epsilon^{-\frac{1}{2}}\min\left\{1, \frac{\lambda^{1}_1(0) -\lambda}{\lambda^1_1(0)}\right\}.
$$
This completes the proof of the lemma.
\end{proof}

The main theorem is an immediate consequence of the following Proposition.

\begin{prop}\label{prop: geomControl}
Suppose $L_0$ has $(1,\epsilon)$-control and let $L_t$ be the solution to the mean curvature flow starting at $L_0$.  Suppose that~\eqref{eq: LocalCurvatureControl} holds. Define 
$$T_{max}:=\sup\{T>0\,:\,\,L_t \subset B_{\Lambda^{-1/2}R}(L_0) \text{  has }(4, 20\epsilon^{\frac1{n+2 }})\text{ control } \forall t\in[0,T)\}.$$ 
If $\epsilon$ is sufficiently small, depending only on background geometry at time $t=0$, then $T_{max}=\infty$, and the mean curvature flow converges exponentially fast to a minimal Lagrangian in $B_{\Lambda^{-1/2}R}(L_0)$ .
\end{prop}
\begin{proof}
 Fix $\epsilon <\frac{1}{2}\min\{1,R^2, \left(\frac{\lambda^{1}_1(0) -\lambda}{\lambda^{1}_1(0)}\right)^2\}$. By the previous lemma  there exists a time $t_1$ so that $L_t$ has $(2 ,10\epsilon)$ control for $t\in[0,t_1)$ and $L_t \subset  B_{\frac{1}{2}\Lambda^{-1/2}R}(L_0)$. Note $\epsilon< \epsilon^{\frac1{n+2 }}$, which implies $0<t_1<T_{max}.$ Now, suppose  that $T_{max}$ is finite. We show  that  $L_t$ actually has $(3, 3\epsilon^{\frac{1}{n+2 }})$ control for  $t\in[\frac{t_1}2, T_{max})$,  and then derive a contradiction.

First choose $\epsilon$ small enough to guarantee  Proposition \ref{exponential} holds  on the interval $ [\frac{t_1}2,T_{max})$, precisely
\begin{equation}\label{eq: epsilonChoice0}
\epsilon^{\frac{1}{n+2}} \leq \frac{1}{C}\min\{1, \frac{a^6}{\Lambda^6}\}\min\left\{1,\frac{ \lambda_1^1(0)-\lambda}{\lambda_1^1(0)}\right\}.
\end{equation}
This implies
$$ \int_{L_t}|H|^2\leq e^{-\frac{at}{80}} \int_{L_{\frac{t_1}2}}|H|^2\leq e^{-\frac{at}{80}}\sup_{L_{\frac{t_1}2}}|H|^2{\rm Vol}(L_{\frac{t_1}2})\leq e^{-\frac{at}{80}}20V\Lambda\epsilon,$$ since $L_{\frac{t_1}2}$ has $(2 ,10\epsilon)$ control. To improve this to $C^0$ exponential decay, note the smoothing estimates of Proposition  \ref{smoothing} imply $|\nabla H|^2\leq\frac{2C\Lambda}{t_1}$, where $C$ depends only on $n$. By making $C$ larger (if necessary),  Lemma \ref{noncollapsed} now gives
$$\sup_{L_t}|H|^2\leq C\max\left\{\frac{ (\frac {\Lambda}{t_1})^{\frac{n}{n+2}}}{\kappa^{\frac 2{n+2}}}, \frac{ \left( V\Lambda\epsilon\right)^{\frac{n}{n+2}}}{ \kappa r_0^n}\right\}\left(e^{-\frac{at}{80}} V\Lambda\epsilon\right)^{\frac2{n+2}}.$$

Set $\alpha=\frac{a}{ 40(n+2)}$, and substituting $t_1= (C\Lambda)^{-1}$ we get
\begin{equation}\label{eq: epsilonChoice1}
\sup_{L_t}|H|^2\leq \Lambda\epsilon^{\frac{1}{n+2}} C\max\left\{ \left(\frac{\Lambda^nV^2\epsilon}{\kappa^2}\right)^{\frac{1}{n+2}}, \frac{ V\epsilon^{\frac{n+1}{n+2}}}{\kappa r_0^n}\right\}e^{-\alpha t}.
\end{equation}
 Choosing 
 \begin{equation}\label{eq: epsilonChoice1a}
 \epsilon \leq \frac{1}{C} \min\left\{\frac{\kappa^2}{\Lambda^nV^2}, \left(\frac{\kappa r_0^n}{V}\right)^{\frac{n+2}{n+1}}\right\}
 \end{equation} yields
\be
\label{expdecayH}
\sup_{L_t}|H|^2\leq \Lambda \epsilon^{\frac{1}{n+2}} e^{-\alpha t}
\ee
for all $t \in[\frac{t_1}2,T_{max}).$ Note, in particular, that this yields an estimate on the Hausdorff distance from $L_t$ to $L_0$.  Precisely, we have
$$
d_{H}(L_t, L_0) \leq C\Lambda^{1/2}\left(\sqrt{\epsilon}t_1 + \epsilon^{\frac{1}{2(n+2)}}\int_{t_1}^{\infty}  e^{-\frac{\alpha}{2} t} dt\right).
$$
Hence we see that $L_t \subset B_{\frac{1}{2}\Lambda^{-1/2}R}(L_0)$, provided
$$
\epsilon \leq \frac{1}{C} \min\left\{R^2, \left(\frac{a}{\Lambda}\right)^{2(n+2)}\right\}.
$$

We can perform the same argument for $|\omega|$. For $t\in[\frac{t_1}2,T_{max})$, combining Proposition \ref{L2B} with Proposition \ref{exponential} gives
$$ 
\begin{aligned}
\int_{L_t}|\omega|^2 &\leq \frac {C}{\lambda^1_1(t)-\lambda}\max\left\{ 1,\frac{ \lambda^1_1(t)}{\lambda^1_1(t)-\lambda}\right\} \int_{L_t} |H|^2\\
&\leq \frac{C}{a}\max\left\{ 1,\frac{ \lambda^1_1(0)}{\lambda^1_1(0)-\lambda}\right\}  e^{-\frac{at}{80}} \int_{L_{\frac{t_1}2}}|H|^2\\
&\leq \frac{C}{a}\max\left\{ 1,\frac{ \lambda^1_1(0)}{\lambda^1_1(0)-\lambda}\right\}V\Lambda\epsilon e^{-\frac{at}{80}} .
\end{aligned}
$$
Recall $|\nabla\omega|\leq 2|h|\leq 4\sqrt{\Lambda}.$  Again for some $C$ depending only on $n$, Lemma \ref{noncollapsed} implies
$$
\begin{aligned}
\sup_{L_t}|\omega|^2\leq& C\max\left\{\frac{ \Lambda^{\frac{n}{n+2}}}{\kappa^{\frac 2{n+2}}}, \frac{ \left(\max\left\{ 1,\frac{ \lambda^1_1(0)}{\lambda^1_1(0)-\lambda}\right\}a^{-1} V\Lambda\epsilon\right)^{\frac{n}{ n+2 }}}{ \kappa r_0^n }\right\}\\
& \cdot \left(\max\left\{ 1,\frac{ \lambda^1_1(0)}{\lambda^1_1(0)-\lambda}\right\}a^{-1}e^{-\frac{at}{80}}V\Lambda\epsilon\right)^{\frac2{n+2}}.
\end{aligned}
$$
Choosing 
\begin{equation}\label{eq: epsilonChoice2}
\epsilon \leq \frac{1}{C}\min\left\{ \kappa^2\left(\frac{a^2}{\Lambda^2}\right)\frac{\min\{1,\frac{\lambda^1_1(0)-\lambda}{\lambda_1^1(0)}\}}{\Lambda^nV^2},  \left(\frac{a}{\Lambda}\cdot \frac{\kappa r_0^n}{V} \cdot \min\left\{1,\frac{\lambda^1_1(0)-\lambda}{\lambda_1^1(0)}\right\} \right)^{\frac{n+2}{n+1}} \right\}
\end{equation}
we obtain 
\be
\label{expdecayomega}
|\omega|^2\leq \epsilon^{\frac{1}{n+2 }} e^{-\alpha t}
\ee
for any time $t \in[\frac{t_1}2,T_{max})$.  In particular, we conclude
$$\sup_{L_t}\Lambda^{-1}|H|^2+\sup_{L_t}|\omega|^2< 3\epsilon^{\frac{1}{n+2}}.$$

Next we turn to the second fundamental form term.  The well-known formula for the time derivative of the second fundamental form along the mean curvature flow (see e.g. \cite[Equation (27)]{Sm3}) yields the following inequality
 $$ \frac{d}{dt} |A|^2\leq 2|\nabla^2 \vec H| |A|+4|A|^3|\vec H|+2|\overline {Rm}| |A| |\vec H|.$$
Because $|\omega|$ is small,  \eqref{equivmetrics} and \eqref{expdecayH} give that $|\vec H|$ is also decaying exponentially. Using this, and  that $L_t$ has $(4, 20\epsilon^{\frac1{n+2 }})$ control, we see
\be
\label{thing400} \frac{d}{dt} |A|^2\leq 2|\nabla^2 \vec H| |A|+C\Lambda^2\epsilon^{\frac1{2(n+2)}}e^{-\frac{\alpha t}2}.
\ee
We would like the entire right hand side to be exponentially decaying. To accomplish this, we use integration by parts and the smoothing estimates of Proposition \ref{smoothing} to see
 \begin{align}
 \int_{L_t}|\nabla^2 \vec H|^2\leq \int_{L_t}|\vec H||\nabla^4 \vec H|\leq \frac{C\Lambda V}{t_1^2} \epsilon^{\frac{1}{2(n+2) }} e^{-\frac{\alpha t}2}.\nonumber
 \end{align}
 Lemma \ref{noncollapsed}, together with $t_1 =\frac{1}{C\Lambda}$ now implies
$$ 
\sup_{L_t}|\nabla^2 \vec H|\leq  C\max\left\{\frac{ \Lambda^{\frac{4n}{2n+4}}}{\kappa^{\frac 1{n+2}}}, \frac{ \Lambda^{\frac{n}{n+2}} \left( {\Lambda V\epsilon^{\frac{1}{ 2(n+2) }} }  \right)^{\frac{n}{2(n+2)}}   }{\sqrt{\kappa r_0^n}}\right\}\left(\Lambda^3 V \right)^{\frac1{n+2}}\epsilon^{\frac{1}{2(n+2)^2 }} e^{-\frac{\alpha t}{2(n+2)}} .$$

 For the time being denote by $B$ the scale invariant constant depending only on background geometry at the initial time, so that 
 $$\sup_{L_t}|\nabla^2 \vec H|\leq \Lambda^{\frac32}B \epsilon^{\frac{1}{2(n+2)^2 }} e^{-\frac{\alpha t}{2(n+2)}}.$$
 Plugging this back into \eqref{thing400} we see
 \be
\frac{d}{dt} |A|^2\leq 2\Lambda^2 B\epsilon^{\frac{1}{2(n+2)^2 }} e^{-\frac{\alpha t}{2(n+2)}}+C\Lambda^2\epsilon^{\frac1{2(n+2)}}e^{-\frac{\alpha t}2}.\nonumber
\ee
Recall that for any constant $c$,  $\int_0^\infty e^{-ct}=c^{-1}$. Thus integrating the above expression in time for $t\in[\frac{t_1}2, T_{max})$, and absorbing constants depending on $n$ into $B$, gives
\begin{align}
 |A|^2(t)&\leq|A|^2(\frac{t_1}2)+ \Lambda^2 B \epsilon^{\frac{1}{2(n+2)^2 }}\alpha^{-1}+C\Lambda^2\epsilon^{\frac1{2(n+2)}}\alpha^{-1}.\nonumber
  \end{align}
We know $|A|^2(\frac{t_1}2)\leq2\Lambda$. As a result, we may choose $\epsilon$ small enough we can ensure $ |A|^2(t)\leq 3\Lambda$ on $[\frac{t_1}2, T_{max})$. More precisely, we require
\be
\label{eq: epsilonChoice3}
\epsilon \leq \frac{1}{C}\min\left\{ \left(\frac{a}{\Lambda}\right)^{2(n+2)^2}\frac{\kappa^{2(n+2)}}{\left(V^2\Lambda^n\right)^{(n+2)}}, \left(\frac{\kappa r_0^n}{V} \cdot \left(\frac{a}{\Lambda}\right)^2\right)^{2(n+2)}\right\}.
\ee

Our next goal is to show that the volume, the eigenvalue $\lambda^1_1$, and the noncollapsing constant $\kappa$ all satisfy the bounds of $L_t$ having $(3,3\epsilon^{\frac{1}{n+2 }})$ control. Note that for $t\in[\frac{t_1}2, T_{max})$
 \begin{align}
 \mu(t)&=\int_0^t2(\sup_{L_s}|A|)(\sup_{L_s}|H|) ds\nonumber\\
 &\leq2\sqrt{3\Lambda}\left(\int_0^{\frac {t_1}2}\sup_{L_s}|H|ds + \int_{\frac {t_1}2}^{t}\sup_{L_s}|H| ds\right)\nonumber\\
 &\leq \sqrt{30\epsilon}\Lambda t_1+6\Lambda  \epsilon^{\frac{1}{2(n+2)}}\int_{\frac {t_1}2}^{t}e^{-\frac{\alpha t}2}\nonumber\\
 &\leq \sqrt{30\epsilon}\Lambda t_1+12\Lambda  \epsilon^{\frac{1}{2(n+2)}} \alpha^{-1}.\nonumber
 \end{align}
 Thus we can choose
 \be
 \label{eq: epsilonChoice4}
 \epsilon \leq \frac{1}{C}\min\left\{ 1, \left(\frac{a}{\Lambda}\right)^{2(n+2)}\right\}\left(\min\left\{1,\frac{\lambda_1^1(0)-\lambda}{\lambda_1^1(0)} \right\} \right)^{2(n+2)}.
 \ee
For an appropriate choice of $C$  it is easy to achieve $e^{-\mu(t)}\geq \frac13$ and so $e^{-\mu(t)}{\rm vol}_{g_0}\leq {\rm vol}_{g_t}\leq e^{\mu(t)}{\rm vol}_{g_0}$  implies
 $$(3V)^{-1}\leq  e^{-\mu(t)}Vol(L_0)\leq  {\rm Vol}(L_t)\leq e^{\mu(t)}{\rm Vol}(L_0)\leq 3V.$$
Similarly, we have seen that if $L_0$ is $\kappa$-noncollapsed at scale $r_0$ at $t=0$, it is $\kappa e^{-(n+1)\mu(t)}$-noncollapsed at scale $r_0$ at later $t$. Thus, again we can choose $C$ so that  $\kappa e^{-(n+1)\mu(t)}\geq 3^{-1}\kappa$
 for $t\in[\frac{t_1}2, T_{max})$. Finally, by equation \eqref{1998}, $\lambda_1^1(t)\geq e^{-3\mu(t)}\lambda_1^1(0)$.  As in the the proof of Lemma \ref{interval}, we can choose $C$ so that the choice of $\epsilon$  from \eqref{eq: epsilonChoice4} implies $\mu(t)$ is controlled by the quantity on the right hand side of \eqref{equation99} for $\delta=2$. Equation \eqref{equation100} now gives
 $$\lambda_1^1(t)-\lambda\geq 3^{-1}a.$$

  Taking everything into account, we have  demonstrated $L_t$  has $(3, 3\epsilon^{\frac{1}{n+2 }})$ control for $t\in[\frac{t_1}2, T_{max})$. By continuity $L_{T_{max}}$  has $(3, 3\epsilon^{\frac{1}{n+2 }})$ control. An application of Lemma \ref{interval} starting at $L_{T_{max}}$ with $\delta=\frac13$ shows there exists a time $t_\delta$, so that if $t\in[T_{max},T_{max}+t_\delta)$, then $L_t$ has  $(4, 20\epsilon^{\frac{1}{n+2}})$-geometry, contradicting the maximality of $T_{max}$. We conclude $T_{max}=\infty.$
  
  As a result the exponential decay \eqref{expdecayH} and \eqref{expdecayomega} holds for $t\in[\frac{t_1}2,\infty)$, and thus $L_t$ converges to a minimal Lagrangian.
  
  \end{proof}
  
  \begin{rk}\label{rk: dependence}
  We can collect the dependencies of $\epsilon$.  Introduce the scale invariant quantities
  $$
  B_0 = \Lambda^nV^2, \quad B_1= \frac{a}{\Lambda},\quad B_2= \frac{\kappa r_0^n}{V},\quad B_3= \min\left\{1, \frac{\lambda_1^1(0)- \lambda}{\lambda_1^1(0)}\right\}.
 $$
 Note that $B_2, B_3 \leq 1$.  Then from ~\eqref{eq: epsilonChoice0}, \eqref{eq: epsilonChoice1}, \eqref{eq: epsilonChoice1a},\eqref{eq: epsilonChoice2}, \eqref{eq: epsilonChoice3}, and \eqref{eq: epsilonChoice4}, we require (after removing redundancies)
 $$
 \begin{aligned}
 \epsilon &\leq \frac{1}{C}\min\left\{R^2,  B_3^{2(n+2)},(B_1^{6}B_3)^{(n+2)}, \frac{\kappa^2}{B_0}, B_2^{\frac{(n+2)}{(n+1)}}, \frac{\kappa^2B_1^2B_3}{B_0}\right\}\\
 \epsilon &\leq\frac{1}{C}\min\left\{ \left(B_1^{2(n+2)}\frac{\kappa^{2}}{B_0}\right)^{(n+2)}, (B_2B_1^{2})^{2(n+2)}, (B_1B_2B_3)^{\frac{(n+2)}{(n+1)}}\right\}.
 \end{aligned}
 $$
 \end{rk}

\section{$L^2$ small initial data}\label{sec: L2main}
In this section we describe the extension of Theorem~\ref{thm: main} to the case of $L^2$ smallness of the initial data.  
\begin{thm}\label{thm: L2main}
Let $(X,\bar g,J,\bar\omega)$ be a K\"ahler-Einstein manifold of complex dimension $n$, and K\"ahler-Einstein constant $\lambda$.  Let $L\subset X$ be a compact, smoothly immersed, totally real submanifold, and if $\lambda=0$ assume the class condition $[\omega] = 0 \in H^{2}(L, \mathbb{R})$.  Suppose that
\begin{equation}\label{eq: quantAlmostReal}
\sup_{L}|\omega|^2<1-\delta.
\end{equation}
In addition suppose that $L$ satisfies the background geometry conditions (i)-(v)  of Section~\ref{sec: intro}.  Then there exists a constant $\epsilon$ depending only on $n$  and lower bounds for $\frac{\kappa r_0^2}{V}$, $\frac{\lambda^{1}_1-\lambda}{\Lambda}$, $\frac1{ \Lambda^nV^2}$, $ \min\{ 1, \frac{\lambda^{1}_1 -\lambda}{\lambda^1_1}\}$, and $R$,  such that, if
$$
\int_{L} \Lambda^{-1/2}|H|^2 + \int_{L}|\omega|^2 \leq \epsilon,
$$
then the mean curvature flow exists for all time and converges exponentially fast to a minimal Lagrangian submanifold.
\end{thm}

\begin{proof}
We sketch the proof.  First note that, as long as $|\omega| <1$ on $L_t$ we know that $L_t$ is totally real.  Thus~\eqref{eq: quantAlmostReal} can be viewed as a quantitative totally real condition.   From~\eqref{eq: 2ndFundEvo} there is a time $t_0= \frac{1}{C\Lambda}\min\{1,R\}$ such that, for all $t\in[0,t_0]$ we have
$$
\sup_{L_t}|A|^2 \leq 2\Lambda,
$$
and $L_t \subset B_{\frac{1}{2}\Lambda^{-1/2}R}(L_0)$.  By the evolution of $|\omega|^2$ we have, as long as $|\omega|<1$, 
$$
\frac{d}{dt}|\omega|^2 \leq \Delta|\omega|^2 + C\Lambda|\omega|^2 +\Lambda\frac{|\omega|^2}{(1-|\omega|^2)}.
$$
The last term on the right hand side arises from the $\eta^{-1}$ in the expression for $dH$ given in Proposition \ref{formuladH}. From the above evolution equation, we can find  a time $t_1 = \frac{1}{C\Lambda }\min\{\delta^2, R\} \leq t_0$ such that $|\omega|^2 \leq (1-\frac{\delta}{2})$ on $[0,t_1]$.  We also see that the $L^2$ norm of $|\omega|$ is controlled for small time.

The evolution of the norm of  mean curvature vector is given by \eqref{mcfvectorevolution}.  After possibly increasing $C$, this equation allows us to see that the $L^2$ norm of $|H|$ can not grow too quickly, and thus for all $t\in[0,t_1]$ we have
$$
\Lambda^{-1/2}\int_{L_t}|H|^2+ \int_{L_t}|\omega|^2 \leq 2\epsilon.
$$
Now,  arguing as in Section~\ref{sec: geomQuant} it is straightforward to check that the non-collapsing, eigenvalue, and volume bounds are preserved up to appropriately scaling them up by $2$, or down by $\frac{1}{2}$, as long as we take $t_2 \leq t_1$ to satisfy
$$
t_2 \leq \frac{1}{C\Lambda}\log\left( \min\left\{2, 2\frac{\lambda_1^1}{\lambda_1^1+\lambda}\right\} \right).
$$

By the smoothing estimates of Proposition \ref{smoothing} together with Lemma~\ref{noncollapsed} we conclude that, for all $t \in [\frac{t_2}{2}, t_2]$
$$
\Lambda^{-1/2}\sup_{L_t}|H|^2+ \sup_{L_t}|\omega|^2 \leq \Psi(\epsilon| \delta, \Lambda, V, \kappa, r_0),
$$
where $\Psi(\epsilon| \delta, \Lambda, V, \kappa, r_0, \lambda_1^1)\rightarrow 0$ goes to zero as $\epsilon \rightarrow 0$, fixing  $\delta, \Lambda, V, \kappa, r_0, \lambda_1^1$. We can now appeal to Theorem~\ref{thm: main} to conclude the result.
\end{proof}

\end{normalsize}

\medskip
\medskip
 

\begin{thebibliography}{99}

{\small
\bibitem{CJL} T. Collins, A. Jacob, and Y.-S. Lin, {\em Special Lagrangian submanifolds of log Calabi-Yau manifolds}, Duke Math. J. {\bf170} (2021), no. 7, 1291--1375.
 
 \bibitem{ChOg} B.-Y. Chen, and K. Ogiue, {\em On totally real submanifolds}, Trans. Amer. Math. Soc. {\bf 193}(1974), 257--266

\bibitem{CY} B. L. Chen and L. Yin, {\it Uniqueness and pseudolocality theorems of the mean curvature flow}. Comm. Anal. Geom. 15 (2007), no. 3, 435--490.

\bibitem{HL} R. Harvey and H. B. Jr. Lawson, {\em Calibrated geometries}, Acta Math.148(1982), 47–157.

\bibitem{Joyce} D. Joyce, {\em Conjectures on Bridgeland stability for Fukaya categories of Calabi-Yau manifolds, special Lagrangians, and Lagrangian mean curvature flow}, EMS Surv. Math. Sci.2(2015), no.1, 1–62.

\bibitem{KT}, K. Kunikawa, and R. Takahashi, {\em Convergence of mean curvature flow in hyper-{K}\"{a}hler manifolds}, Pacific J. Math. {\bf 305}(2020), no. 2, 667--691

\bibitem{Lees} J. A. Lees, {\em On the classification of {L}agrange immersions}, Duke Math. J.  {\bf 43} (1976), no. 2, 217--224
   
\bibitem{Li} H. Li, {\it Convergence of Lagrangian mean curvature flow in K\"ahler-Einstein manifolds}. Math. Z. {\bf 271} (2012), no. 1-2, 313--342. 
\bibitem{LP1} J. D. Lotay and T. Pacini {\it From Lagrangian to totally real geometry: coupled flows and calibrations}, Communications in Analysis and Geometry 28 (2020), 607--675.
\bibitem{LP2} J. D. Lotay, and T. Pacini, {\em Complexified diffeomorphism groups, totally real submanifolds and {K}\"{a}hler-{E}instein geometry}, Trans. Amer. Math. Soc. {\bf 271} (2019), no. 4, 2665--2701.
\bibitem{LP3} J. D. Lotay, and T. Pacini, {\em From minimal {L}agrangian to {$J$}-minimal submanifolds: persistence and uniqueness}, Boll. Unione Mat. Ital. {\bf 12} (2019), no. 1-2. 63--82

\bibitem{LSch} J. D. Lotay, and F. Schulze, {\em Consequences of strong stability of minimal submanifolds}, Int. Math. Res. Not. IMRN (2020), no. 8, 2352--2360.

\bibitem{LSS} J. D. Lotay, F. Schulze, and G. Sz\'ekelyhidi, {\em Neck pinches along the Lagrangian mean curvature flow of surfaces}, preprint, arXiv:2208.11054

\bibitem{LSS2} J. D. Lotay, F. Schulze, and G. Sz\'ekelyhidi, {\em Ancient solutions and translators of Lagrangian mean curvature flow}, preprint, arXiv: 2204.13836



\bibitem{McLean} R. McLean, {\em Deformations of calibrated submanifolds}, Comm. Anal. Geom.6 (1998), no.4, 705–747.
\bibitem{Neves} A. Neves, {\em Singularities of Lagrangian mean curvature flow: zero-Maslov class case}, Invent. Math.168(2007), no.3, 449–484.
\bibitem{Neves1} A. Neves, {\em Finite time singularities for {L}agrangian mean curvature flow}, Ann. of Math. (2) {\bf 177} (2013), no. 3, 1029--1076
\bibitem{Oh} Y.-G. Oh, {\em Second variation and stabilities of minimal {L}agrangian  submanifolds in {K}\"{a}hler manifolds}, Invent. Math. {\bf 101} (1990), no. 2, 501--519.

\bibitem{Pacini} T. Pacini, {\em Extremal length in higher dimensions and complex systolic inequalities}, J. Geom. Anal. {\bf 31} (2021), no. 5, 5073--5093.

\bibitem{PS} D. H. Phong, and J. Sturm {\em On stability and the convergence of the {K}\"{a}hler-{R}icci flow}, J. Differential Geom. {\bf 72} (2006), no. 1, 149--168.
\bibitem{PSSW} D. H. Phong, J. Song, J. Sturm, and B. Weinkove {\em The {K}\"{a}hler-{R}icci flow and the {$\overline{\partial}$} operator on vector fields}, J. Differential Geom. {\bf 81} (2009), no. 3, 631--647
\bibitem{Sm1} K. Smoczyk, {\it A canonical way to deform a Lagrangian submanifold}. Preprint. arXiv:dg-ga/9605005.
\bibitem{Sm2} K. Smoczyk, {\it Angle theorems for the Lagrangian mean curvature flow}. Math. Z. 240 (2002), no. 4, 849--883




 \bibitem{Sm3} K. Smoczyk {\it Mean curvature flow in higher codimension: introduction and survey}. Global differential geometry, 231-274, Springer Proc. Math., 17, Springer, Heidelberg, 2012.



\bibitem{Thomas} R. P. Thomas, {\em Moment maps, monodromy and mirror manifolds}, Symplectic geometry and mirror symmetry (Seoul, 2000), 467–498.
World Scientific Publishing Co., Inc., River Edge, NJ, 2001.


\bibitem{TY} R. Thomas and  S.-T. Yau, {\it Special Lagrangians, stable bundles and mean curvature flow.} Comm. Anal. Geom. 10 (2002), no. 5, 1075-1113. 

\bibitem{TsaiWangStab} C.-J. Tsai and M.-T. Wang, {\em A strong stability condition on minimal submanifolds and its implications}, J. Reine Angew. Math. {\bf 764}(2020), 111--156

\bibitem{TsaiWangAH} C.-J. Tsai, and M.-T. Wang, {\em Global uniqueness of the minimal sphere in the {A}tiyah-{H}itchin manifold}, Math. Res. Lett., {\bf 29} (2022), no. 3, 871--886.

\bibitem{TsaiWangHol} C.-J. Tsai, and M.-T. Wang, {\em Mean curvature flows in manifolds of special holonomy}, J. Differential Geom. {\bf 108} (2018), no. 3, 531--569

\bibitem{White} B. White {\em The space of minimal submanifolds for varying {R}iemannian metrics}, Indiana Univ. Math. J. {\bf 40} (1991), no. 1, 161--200.

\bibitem{Yau} S.-T. Yau, {\em Submanifolds with constant mean curvature. {I}, {II}}, Amer. J. Math. {\bf 96} (1974), 346--366; ibid. {\bf 97 (1975), 76--100}

}





\end{thebibliography}
\end{document}